\documentclass[a4paper,11pt]{article}

\usepackage[latin,english]{babel}
\usepackage[utf8]{inputenc}
\usepackage[T1]{fontenc}
\usepackage{eufrak, color, mathrsfs,dsfont}
\usepackage[letterpaper,margin=.96in,bottom=1in,top=.8in]{geometry}
\usepackage{amsmath}
\usepackage{amssymb}
\usepackage{amsthm}
\usepackage{enumitem}
\usepackage{mathtools}
\usepackage{mathabx}
\usepackage{cases}
\usepackage{braket}
\usepackage[toc,page]{appendix}
\usepackage{pdfsync}
\usepackage{hyperref}
\usepackage{psfrag}
\usepackage[mathscr]{euscript}

\usepackage{authblk}

\everymath{\displaystyle}
\usepackage{amssymb,amsmath,amsthm}
\usepackage{graphicx}

\def\R{\mathbb R}
\usepackage{amsfonts}
\usepackage{latexsym}
\usepackage{amsopn}
\usepackage{amscd}
\usepackage{mathrsfs}
\usepackage{dsfont}
\usepackage[english]{babel}
\usepackage{caption}

\numberwithin{equation}{section}

\font\teneufm=eufm10
\font\seveneufm=eufm7
\font\fiveeufm=eufm5
\newfam\eufmfam
\textfont\eufmfam=\teneufm
\scriptfont\eufmfam=\seveneufm
\scriptscriptfont\eufmfam=\fiveeufm

\newtheorem{theorem}{Theorem}[section]
\newtheorem{proposition}{Proposition}[section]
\newtheorem{lemma}{Lemma}[section]

\newtheorem{remark}{Remark}[section]
\newtheorem{remarks}{Remark}[section]
\newtheorem{definition}{Definition}[section]
\newcommand{\be}{\begin{equation}}
\newcommand{\ee}{\end{equation}}

\usepackage{color}

\title{On the global well-posedness of the 3D axisymmetric resistive  MHD
equations}
\author{Zineb  Hassainia}

\newcommand{\Addresses}{{
  \bigskip
  \footnotesize

  Z.~Hassainia, \textsc{NYU, Abu Dhabi
Saadiyat Marina District - Abu Dhabi, United Arab Emirates.}\par\nopagebreak
  \textit{E-mail address}: \texttt{zh14@nyu.edu}

}}

\begin{document}

\date{}

\maketitle

\begin{abstract}
In this paper, we prove  the global well-posedness for the  three-dimensional magnetohydrodynamics (MHD) equations with zero viscosity and axisymmetric initial data. First, we analyze  the
problem corresponding to the Sobolev  regularities $ H^s\times H^{s-2}$, with $  s>5/2$. Second, we address the same problem but for the Besov critical regularities $ B_{p,1}^{3/p+1}\times B_{1,p}^{3/p-1}$, $2\leq p\leq \infty$. This case turns out to be more subtle as  the Beale-Kato-Majda criterion is not known to be valid for rough regularities.
\end{abstract}

\tableofcontents

\section{Introduction }
In this paper, we consider  the three-dimensional   incompressible magnetohydrodynamic (MHD) system, describing  the motion of an electrically conducting fluid with a neglected  viscosity and an  important  resistivity,
\begin{equation}
\label{MHD}
\begin{cases}
\partial_t  v +  v \cdot \nabla  v + \nabla p=b\cdot\nabla b, \quad t\in \R_+, \, x\in \R^3,   \\
\partial_t  b +  v \cdot \nabla  b - \Delta b=b\cdot\nabla v,  \\
{\rm div}\,   v = 0\, ,\,  {\rm div}\, b=0,  
\end{cases}
\end{equation}
where $v$ denotes the fluid velocity  and $b$ stands for the magnetic field. The pressure $p$ is a scalar function that can be recovered from the velocity and the magnetic field by inverting an elliptic equation. Note that we have set   the magnetic diffusion coefficient to be one  by a scaling transformation. 
For a general review about the derivation of the MHD equations we refer the reader to \cite{B,Ch,Dav}.  
\vspace{0.2cm}

Since the pioneering work of  Alfv\'en \cite{A}, the MHD equations have played prominent roles  in the study of significant phenomena such as the magnetic reconnection in astrophysics,  geomagnetic dynamo in geophysics and   plasma confinement  in  engineering, see \cite{B,Dav}.
From mathematical point of view, the MHD system and its ideal,  fully dissipative and  partially dissipative counterparts have been intensively investigated in the few last decades and considerable interest has been devoted to   the  local/global well-posedness problem, see for instance \cite{DL,ST, CK,CST,CMZ,CMRR, CW,FMRR,FMRR2,JN,MY,Sc,Se,W} and the references therein.
Note that the global existence and the uniqueness  of such solutions is an open problem except in the two dimensional case with the full dissipation. 
\vspace{0.2cm}

 Hereafter, we shall primarily restrict the discussion to the MHD equations with only  magnetic diffusion \eqref{MHD}. 
 In \cite{K}, Kozono  proved global existence of weak solutions in 2D for  initial data in $L^2$; while in 3D, Fan and Ozawa \cite{FO}  showed that  the solution can be extended beyond time $T$ if $\nabla v\in L^1(0,T;L^\infty(\R^3))$. 
 However, it is not clear if such weak solutions are unique or  if one can extend the the global weak theory, in 2D, to the classical one. 
 This latter problem is  critical in the sense that if we replace the Laplacian term $-\Delta$ by a fractional Laplacian operator $(-\Delta)^\beta$,  $\beta>1$, the resulting system then admits a unique global solution. This fact was proved independently in \cite{CWY,JZ}. 
We stress that  the crucial part, to prove the global well posedness for \eqref{MHD} in 2D,  is to get an a priori estimate for the vorticity $\omega:=\textnormal{curl}( v)$ in $L^\infty$.  The main difficulty can be illustrated through the equations governing the vorticity $\omega:=\textnormal{curl}( v)=\partial_1 v_2-\partial_2 v_1$ and the current density $j:=\textnormal{curl}( b)=\partial_1 b_2-\partial_2 b_1$, 
\begin{equation}
\begin{cases}
\partial_t   \omega+ v \cdot \nabla  \omega= b\cdot\nabla j,\\
\partial_t   j+ v \cdot \nabla  j-\Delta j=b\cdot\nabla j+2\partial_1 b\cdot\nabla v_2-2\partial_2 b\cdot\nabla v_1.
\end{cases}
\end{equation}
According to the first equation in the above system, the vorticity conservation requires a global bound of $\|\nabla j\|_{L\infty}$. However,  the nonlinear structure of the term $\mathcal{Q}( \nabla  v,\nabla  b):=2\partial_1 b\cdot\nabla v_2-2\partial_2 b\cdot\nabla v_1$ and the lack of continuity of Riesz transforms on the bounded functions makes the the problem  highly non trivial. 
For regularity criteria we refer the reader for instance to \cite{CK,LZ,JZ}.
\vspace{0.2cm}

The main scope of this paper is  to construct global unique classical  solutions in 3D under some geometrical constraints. More precisely, we shall see that the axisymmetry  offer a suitable class of initial data for which the construction of classical  solutions is possible. Note that this  has been successfully implemented for the viscous non-resistive  MHD equations,   by Lei in  \cite{Lei} where he proves the existence of  axisymmetric global  classical solutions of the special form
\be\label{spform}
v(t,x) = v^{r}(t,r, z)e_{r} + v^{3}(t,r, z)e_{z},\quad b=b^\theta (t,r, z)e_{\theta} 
\ee
where  $r=(x_{1}^{2}+x_{2^{2}})^{\frac{1}{2}}$,  $(e_{r},e_{\theta},e_{z})$ is the cylindrical basis of $\R^3$.
We stress that the fluid velocity $v$ and the magnetic field $b$   are assumed to be invariant by rotation around the vertical axis. 
\vspace{0.2cm}
  
  Before stating our  contribution in the subject let us briefly discuss what is known for Euler equations ( obtained by setting  $b$  identically zero)  with this special initial data.  
  In \cite{ui68}, Ukhovskii and Yudovich 
proved the global existence for axisymmetric initial data with finite energy and satisfying
in addition $\omega_0\in L^2\cap L^\infty$ and ${\omega_0}/{r}\in L^2\cap L^\infty$. 
This result has been improved by Shirota and Yanagisawa \cite{SY} who proved global existence
in $H^s$ , with $s >5/2$. In  \cite{d07}, Danchin has obtained global existence
and uniqueness for initial data $\omega_0\in L^{3,1}\cap L^\infty$ and ${\omega_0}/{r}\in L^{3,1}$, where $L^{3,1}$  denotes the Lorentz space. The global well-posedness in the critical Besov regularity, that is, $v_0\in B_{p,1}^{3/p+1}$, $p\in [1,\infty]$ with ${\omega_0}/{r}\in L^{3,1}$  was established in \cite{AHK}. We point out that the global well-posedness result has been extended  for 3D   Euler–Boussinesq system, which couples Euler equations with a transport-heat equation governing the density, see \cite{HR,S}.
  
Our main concern here is to validate similar  global well-posedness results for the MHD system \eqref{MHD}. Our first result reads as follows

\begin{theorem}\label{thm:sub-critical}
Let $v_0$  and $b_0$ be two axisymmetric divergence-free vector fields as in \eqref{spform}.
Suppose that $(v_0,b_0)\in H^s\times H^{s-2}$,  with $s >5/2$,  and   ${b_0^\theta}/{r}\in L^2 \cap L^{\infty}$. Then there exists a unique global solution $(v,b)$ to the MHD system \eqref{MHD} such that
$$
(v,b)\in \mathcal{C}\big(\R_+; H^s\big)\times \Big(\mathcal{C}\big(\R_+; H^{s-2}\big)\cap \widetilde{L}^1_{\textnormal{loc}}\big(\R_+; H^s\big)\Big) \quad \textnormal{and}\quad \frac{b}{r}\in L^\infty_{\textnormal{loc}}\big(\R_+; L^2 \cap L^{\infty}\big).
$$

\end{theorem}

Our second main result deals with the critical Besov spaces.

\begin{theorem}\label{thm:critical}
Let $p\in [2,\infty]$,    $v_0\in L^2\cap B_{{p},1}^{3/p+1}$ and  $b_0\in L^2$ be two axisymmetric divergence-free vector fields as in \eqref{spform}. Assume, in addition, that  
${\omega_0^\theta}/{r}\in L^{3,1}$ and ${b_0^\theta}/{r}\in L^2 \cap L^{\infty}$, where $\omega_0^\theta$  is the angular component of the vorticity $\omega_0:=\textnormal{curl}( v_0)$.
\begin{itemize}
\item Case $p=\infty$. If $b_0\in  B_{\sigma,1}^{3/\sigma-1}$,  $\sigma\in [2,\infty)$,
then there exists a unique global solution $(v,b)$ to the  system \eqref{MHD} such that
$$
(v,b)\in \mathcal{C}\big(\R_+; B_{\infty,1}^{1}\big)\times \Big(\mathcal{C}\big(\R_+;  B_{p,1}^{3/\sigma-1} \big)\cap {L}^1_{\textnormal{loc}}\big(\R_+;  B_{\infty,1}^{1}\big)\Big). 
$$
\item Case $p<\infty$. If $b_0\in  B_{p,1}^{3/p-1}$ 
then there exists a unique global solution $(v,b)$ to the  system \eqref{MHD} such that
$$
(v,b)\in \mathcal{C}\big(\R_+; B_{p,1}^{\frac3p+1}\big)\times \Big(\mathcal{C}\big(\R_+;  B_{p,1}^{3/p-1} \big)\cap {L}^1_{\textnormal{loc}}\big(\R_+;  B_{p,1}^{3/p+1}\big)\Big).
$$
\end{itemize}
Moreover, in both cases, one has
$$
\Big(\frac{\omega}{r},\frac{b}{r}\Big)\in L^\infty_{\textnormal{loc}}\big(\R_+;  L^{3,1}\big)\times L^\infty_{\textnormal{loc}}\big(\R_+; L^2 \cap L^{\infty}\big).
$$
\end{theorem}
A few remarks are in order
\begin{remark}~

\begin{enumerate}[label=\rm(\roman*)]
\item The condition $\omega_0/r\in L^{3,1}$ is automatically derived from $v_0\in B_{p,1}^{1+3/p}$ for $p<3$,  see \eqref{emlb}.
\item For $p > 2$ the assumption $v_0\in L^{2}$ is still needed because  the only energy estimate owned for the MHD system is in $L^2$.
\item In the limiting case $p=\infty$,  the critical space for magnetic field is $B_{\infty,1}^{-1}$. However,   we were unable to get  smoothing effects on this space  and  had to relax with the assumption $b_0\in B_{\sigma,1}^{3/\sigma-1}$, for some $\sigma<\infty$.

\item Note that if  $\sigma\in(3,\infty)$ then we  have  $b\in  {L}^1_{\textnormal{loc}}\big(\R_+;  B_{\sigma,1}^{3/\sigma+1}\big)$.
 \end{enumerate}
\end{remark}

Now, we shall briefly  discuss the main ideas of the proof.
Note that, in space dimension three, the vorticity  $\omega:=\textnormal{curl}( v)$ satisfies the equation
\begin{equation}\label{om-int}
\partial_t   \omega+ v \cdot \nabla  \omega= \omega\cdot\nabla v+\textnormal{curl}(b\cdot\nabla b)
\end{equation}
and the way to control the vortex stretching term  in the right-hand side is a widely open
problem even for trivial magnetic fields. Nevertheless, for axisymmetric flows, as in \eqref{spform}, the
vorticity has the special form
$$
\omega=\omega^\theta e_\theta
$$
and the equation 
 \eqref{om-int} can be written as
 $$
\partial_t   \omega+ v \cdot \nabla  \omega=\frac{v^r}{r}\omega-\frac{\partial_z (b^\theta)^2}{r}e_\theta.
 $$
Thus, by dividing  the vorticity equation by $r$, one can absorb the vortex stretching term into the convection term, leaving only one term involving the quatity $ \Gamma:={b^\theta}/{r}$  as a forcing one. More precisely,   the scalar function $\Omega:={\omega^\theta}/{r}$ satisfies
\begin{equation}
\label{eq:om-int}
\partial_t  \Omega +  v \cdot \nabla\Omega=-\partial_z\Gamma^2.
\end{equation}
In the light of the proof for the 3-D axisymmetric Euler equations, the key point is to get a control of the Lorentz norm $\|\Omega\|_{L^{3,1}}$. This,  according to \eqref{eq:om-int},
  requires some strong a priori estimates on $ \Gamma$.  Since the magnetic fields has the spacial form \eqref{spform} then  the magnetic stretching term in the second equation of \eqref{MHD} reads
 $$
b\cdot\nabla v=\frac{v^r}{r}b^\theta,
 $$
and  the quantity $ \Gamma:=\frac{b^\theta}{r}$ 
solves the equation
 $$
 \partial_t  \Gamma + v \cdot \nabla \Gamma - \Big(\Delta +\frac{2}{r}\partial_r\Big)\Gamma =0 .
 $$
Therefore,  by using   the  smoothing effect
of the heat flow we obtain the desired control on $\Gamma$ and then on $\Omega$. This  allows to get  the $L^\infty$ bound of the vorticity $\omega$  for every time through the Biot-Savart law.  Then, in order to globally propagate the optimal subcritical regularities, in the scaling sense, of the MHD system \eqref{MHD} we uses some  refined informations about the axisymmetric geometry  of the magnetic field $b$. 
\vspace{0.2cm}

The situation is more complicate for the critical Besov regularities  since the Beale-Kato-Majda criterion significant quantity that one should estimate is $\|\omega\|_{B^0_{\infty,1}}$. For this aim we use the approach developed in \cite{AHK}   where the axisymmetric geometry plays a crucial role. This allows to bound for every time the Lipschitz norm of the velocity and then to propagate the regularities.
\vspace{0.2cm}

The rest of this paper is organized as follows. In section 2, we recall some function spaces and give some of their useful properties, we also gather some preliminary estimates. In Section 3  we shall be concerned with  some a priori estimate. In section
5 and 6 we give respectively the proof of Theorem 1.1 and Theorem  1.2. In the appendix, we establish
some product laws for axisymmetric flows.

\section{Tools and function spaces}

 \quad  Throughout this paper, $C$ stands for some real positive constant which may be different in each occurrence and $C_0$ for a positive constant depending on the size of the  initial data. We shall sometimes alternatively use the notation $X\lesssim Y$ for an inequality of the type $X \leq CY$. Also, for any pair of operator $D$ and $F$ on some Banach space $\mathcal{A}$, the commutator $[D,F]$ is defined by $DF-FD.$\\
 We shall  denote by
$$\Phi_{l}(t)=C_{0}\underbrace{\exp(...\exp}_{l-times}(C_{0} t^{\frac54})...),$$ where $C_{0}$ depends on the initial data and its value may vary from line to line up to some absolute constants. We will also make an intensive use of the following trivial facts
\begin{equation*}
\int^{t}_{0}\Phi_{l}(\tau) d\tau \le \Phi_{l}(t)\qquad\textnormal{and}\qquad \exp\Big(\int^{t}_{0}\Phi_{l}(\tau) d\tau\Big)\le \Phi_{l+1}(t).
\end{equation*}

\subsection*{Littlewood-Paley theory} Let us recall briefly the classical dyadic partition of the unity, for a proof see for instance \cite{BCD}: there exists two positive radial functions $\chi\in\mathcal{D}(\R^3)$ and $\varphi\in\mathcal{D}(\R^3\backslash\{0\})$ such that
\begin{equation*}\label{3}
\forall\xi\in\R^3,\quad \chi(\xi)+\sum_{q\ge0}\varphi(2^{-q}\xi)=1,
\end{equation*}
\begin{equation*}\label{4}
\forall\xi\in\R^3\backslash\{0\},\quad \sum_{q\in\mathbb{Z}}\varphi(2^{-q}\xi)=1,
\end{equation*}
\begin{equation*}\label{6}
\vert j-q\vert\ge 2\Rightarrow \quad \textnormal{Supp}\ \varphi(2^{-j}.)\cap \textnormal{Supp}\ \varphi(2^{-q}.)=\varnothing,
\end{equation*}
\begin{equation*}\label{7}
q\ge1\Rightarrow\quad \textnormal{Supp}\ \chi\cap\textnormal{Supp}\ \varphi(2^{-q}.)=\varnothing.
\end{equation*}
 For every $v \in\mathcal{ S}'(\R^3)$ one defines the non-homogeneous Littlewood-Paley operators by,
\begin{align*}
\Delta_{q}v &=\varphi(2^{-q}\textnormal{D})v= 2^{3\,q}h (2^{q}.) \ast v\qquad\textnormal{for}\qquad q\geqslant 0,\\
S_{q}v&=\chi(2^{-q}\textnormal{D})v =\displaystyle \sum_{-1\le p\le q-1}\Delta_{p}v= 2^{3q} g(2^{q}.) \ast v ,\\
\Delta_{-1}v&=S_{0}v,\qquad \Delta_{q}v=0 \qquad \textnormal{for}\qquad q\le-2.
\end{align*}
with  $h=\mathcal{F}^{-1}\varphi$ and $g=\mathcal{F}^{-1}\chi.$ Similarly, we define the homogeneous operators by
$$
\forall{q}\in \mathbb{Z}\quad \dot{\Delta}_{q} v=\varphi(2^{-q}\textnormal{D})v\quad\textnormal{and}\quad \dot{S}_{q}v=\sum_{-\infty\le j\le q-1}\dot{\Delta}_{j}v.
$$
 One can easily check that for every tempered distribution $v$, we have
	$$
	v=\sum_{q\geq -1}\Delta_qv,
	$$
		and for all $v\in\mathcal{ S}'(\R^3)/\{\mathcal{P}[\R^3]\}$ 
	$$
	v=\sum_{q\in\mathbb{Z}}\dot{\Delta}_qv,
	$$
	where $\mathcal{P}[\R^3]$ is the space of polynomials, see \cite{Pee}.
Furthermore,  the Littlewood-Paley decomposition satisfies the property of almost orthogonality
for any $u,v\in\mathcal{S}^{\prime}(\R^3),$
\begin{equation}\label{orth}
\begin{split}
\Delta_{p}\Delta_{j}u&=0\quad \textnormal{if} \qquad \vert p-j \vert \geqslant 2,\\
\Delta_{p}(S_{j-1}u\, \Delta_{j}v)&=0 \quad \textnormal{if} \qquad  \vert p-j \vert \geqslant 5.
\end{split}
\end{equation}
	The following lemma  describes how the derivatives act on spectrally localized functions, see for instance \cite{che98}.
\begin{lemma}[Bernstein inequalities]\label{br}
There exists a constant $C > 0$ such that for all $q,k  \in \mathbb{N}, 1\leq a\leq b\leq\infty$ and for every tempered
distribution $u$ we have
$$
\sup_{\vert\alpha\vert\leq k}\Vert\partial^\alpha S_q u\Vert_{L^b}\leq C^k2^{q(k+3(\frac{1}{a}-\frac{1}{b}))}\Vert S_qu\Vert_{L^a},
$$
$$
C^{-k}2^{qk}\Vert \dot{\Delta}_qu\Vert_{L^b}\leq\sup_{\vert\alpha\vert= k}\Vert\partial^\alpha \dot{\Delta}_qu\Vert_{L^b}\leq C^k2^{qk}\Vert \dot{\Delta}_qu\Vert_{L^b}.
$$
\end{lemma}
\subsection*{Besov spaces}
Based on Littlewood-Paley operators, we can define  Besov spaces as follows. Let $(p, r) \in [1, +\infty]^2$ and $s \in \R$. The non-homogeneous Besov space $B_{p,r}^s$ is the set of tempered distributions $v$ such that
$$
\Vert v\Vert_{B_{p, r}^{s}}\triangleq\Big\Vert\big(2^{qs}\Vert\Delta_{q}v\Vert_{L^p}\big)_{q\in\mathbb{Z}}\Big\Vert_{\ell^{r}(\mathbb{Z})}<+\infty.
$$
The homogeneous Besov space $\dot{B} ^s_{p,r}$ is defined as the set of $\mathcal{ S}'(\R^3)/\{\mathcal{P}[\R^3]\}$  such that
$$
\Vert v\Vert_{\dot{B}_{p, r}^{s}}\triangleq\Big\Vert\big(2^{qs}\Vert\dot{\Delta}_{q}v\Vert_{L^p}\big)_{q\in\mathbb{Z}}\Big\Vert_{\ell^{r}(\mathbb{Z})}<+\infty.
$$
We point out that, for all $s\in\R$ the Besov space $B_{2,2}^s$ coincides with the Sobolev space $H^s$.
The following embeddings are an easy consequence of Bernstein inequalities,
$$
B_{p_1, r_1}^{s}\hookrightarrow B_{p_2, r_2}^{s+3(\frac{1}{p_2}-\frac{1}{p_1})} \quad p_1\leq p_2\quad \textnormal{and}\quad r_1\leq r_2.
$$
In order to obtain a better description of the regularizing effect of the transport-diffusion equation, we need to use Chemin-Lerner type spaces $\widetilde L^{\rho}_{T}B^{s}_{p,r}$, see for instance \cite{BCD}.
Let $T>0,$ $\rho\ge 1,$ $(p,r)\in[1,\infty]^{2}$ and $s\in\R,$ we denote by $L_{T}^{\rho}B_{p,r}^{s}$ the space of distribution $f$ such that
$$\Vert f\Vert_{L_{T}^{\rho}B_{p,r}^{s}}:=\Big\Vert \Big(2^{js}\Vert\Delta_{j}f\Vert_{L^p}\Big)_{\ell^r}\Big\Vert_{L_{T}^{\rho}}<+\infty.$$
We say that $f$ belongs to the Chemin-Lerner space $\widetilde{L}_{T}^{\rho}B_{p,r}^{s}$ if 
$$\Vert f\Vert_{\widetilde{L}_{T}^{\rho}B_{p,r}^{s}}:=\Big\Vert 2^{js}\Vert\Delta_{j}f\Vert_{L_{T}^{\rho}L^p}\Big\Vert_{\ell^r}<+\infty.$$
The relation between these spaces are detailed in the following lemma, which is a direct consequence of the  Minkowski inequality. 
\begin{lemma}\label{le1}
 Let $ s\in\R ,\varepsilon>0$ and $(p,r,\rho) \in[1,+\infty]^3.$ Then we have the following embeddings
$$L^{\rho}_{T}B^{s}_{p,r}\hookrightarrow\widetilde L^{\rho}_{T}B^{s}_{p,r}\hookrightarrow L^{\rho}_{T}B^{s-\varepsilon}_{p,r}\;\;\;\textnormal{if}\quad r\geqslant\rho.$$ 
$${L^\rho_{T}}{B_{p,r}^{s+\varepsilon}}\hookrightarrow\widetilde L^\rho_{T}{B_{p,r}^s}\hookrightarrow L^\rho_{T}B_{p,r}^s\;\;\;\textnormal{if}\quad 
\rho\geq r.$$
\end{lemma}

Now, we introduce the Bony's decomposition \cite{bo81} which is  the basic tool of the para-differential calculus. 
Formally the product of two tempered distributions  $u$ and $v$ can be  divided into three parts  as follows:
\be\label{j}
uv=T_uv+T_vu+R(u,v),
\ee
where
$$
T_uv=\sum_qS_{q-1}u\,\Delta_qv\quad\textnormal{and}\quad R(u,v)=\sum_q\Delta_qu\,\widetilde{\Delta}_qv,\quad \textnormal{with}\quad\widetilde{\Delta}_q =\sum_{i=-1}^1\Delta_{q+i}.
$$
\subsection*{ Some commutator estimates}
The following proposition is proved in \cite{HK}.
\begin{proposition}\label{prop:commu}
Let $u$ be a smooth function and $v$ be  a smooth divergence-free vector field of $\mathbb{R}^3$  such that it's vorticity $\omega:=\textnormal{curl}\, v$ belongs to $L^\infty$. Then for every $p\in[1,\infty]$ and  $q\geq -1$,  we have
$$
\Vert [\Delta_q,v \cdot \nabla]u\Vert_{L^p}\lesssim \Vert u\Vert_{L^p}\Big(\Vert \nabla\Delta_{-1}v\Vert_{L^\infty}+(q+2)\Vert \omega\Vert_{L^\infty}\Big).
$$
\end{proposition}
\noindent The following proposition is proved in \cite{HR}
\begin{proposition}\label{prop:commu2}
Let  $v$ be  an axisymmetric smooth and  divergence-free vector field without swirl  and let $u$ be a smooth scalar function.   
Then there exists $C>0$ such that  for every $q\in \mathbb{N}\cup \{-1\}$,  we have
$$
\big\Vert [\Delta_q,v \cdot \nabla]u\big\Vert_{L^2}\leq C\Big\Vert \frac{\omega^\theta}{r}\Big\Vert_{L^{3,1}}\Big(\Vert x_h u\Vert_{L^6}+\Vert u\Vert_{L^2}\Big).
$$
where $\omega^\theta$ is the angular component of $\omega=\nabla\times v$.
\end{proposition}
\noindent The following  commutator estimate is  proved in  \cite{che98},
\begin{lemma} \label{m1}
Let $\eta$ be a smooth function and $v$ be a smooth vector field of $\R^{3}$ with zero divergence. Then we have, for all $1 \le p \le \infty$  and $-1<s< 1$,
\begin{equation*}
\sum_{j \ge -1}2^{js}\Vert[\Delta_{j}, v\cdot\nabla]u\Vert_{L^p}\lesssim 
\Vert\nabla v\Vert_{L^\infty}\Vert u\Vert_{B_{p,1}^s}.
\end{equation*}
\end{lemma} 
\subsection*{Lorentz spaces and interpolation} 
Let $p\in ]1,\infty[$, $q\in [1,\infty]$, The Lorentz space $L^{p,q}$ can be defined by real interpolation from Lebesgue spaces:
$$
(L^{p_{0}} , L^{p_{1}})_{(\theta,q)} = L^{p,q}
$$
where $1\leq p_{0}<p < p_1\leq \infty$ and $\theta$ satisfies $\dfrac{1}{p}=\dfrac{1-\theta}{p_{0}}+\dfrac{\theta}{p_{1}}$ with $0<\theta<1$.
The spaces $L^{p,q}$ have the following properties: (see pages 18-20 of \cite{L})
\begin{align}\label{lor}
\nonumber& L^{p,p} = L^{p}, \\ &\nonumber L^{p,q} \hookrightarrow L^{p,q'},\quad 1\leq q\leq q'\leq \infty,\\
& \Vert uv\Vert_{L^{p,q}}\leq \Vert u\Vert_{L^{\infty}}\Vert v\Vert_{L^{p,q}}.
 \end{align}
 Furthermore, for all $1\leq p<3$ we have the embedding, see page 21 of \cite{AHK}, 
 \be\label{emlb}
 B_{p,1}^{\frac3p-1}\hookrightarrow L^{3,1}.
 \ee
\subsection*{Estimates for a transport equation}
We need the following proposition which deals with the persistence of
Besov regularities in a transport equation. Abidi, Hmidi and Keraani, in \cite{AHK}, proved the result for the case $s=\pm1$. The remainder cases are done in \cite{che98, vis98}.
\begin{proposition}[Proposition A.2 in \cite{AHK}]\label{prop:trans-besov} Let $ s \in(-1,1)$, $ p,r \in[1,\infty]$ and $v$ be a smooth divergence-free vector field. Let $f$ be a smooth solution of the equation
\begin{equation*}
\partial_t f+v\cdot\nabla f=g,  \quad
f_{|t=0}=f_{0},
\end{equation*}
where $f^0 \in B_{p,r}^s(\R^3)$ and $g \in L^1_{loc}(\R_+,B_{p,r}^s).$ Then for all $ t \in \R_+,$ we have 
\begin{equation*}
\Vert f(t) \Vert_{B_{p,r}^s}\lesssim e^{C U(t)}\Big(\Vert f^0 \Vert_{B_{p,r}^s}+\int_0^t e^{-C U(\tau)}\Vert g(\tau) \Vert_{B_{p,r}^s} d\tau\Big),
\end{equation*}
with
$U(t):=\Vert\nabla v(t) \Vert_{L_t^1L^\infty}$ and $C$ is a constant depending on $s$.   The above estimate is true in the two following case
$$s=-1,\, r=\infty, 1 \le p \le \infty\quad\textnormal{and}\quad  s=r=1,\,  1 \le p \le \infty,$$
 provided that we change $U(t)$ by $V(t):=\Vert v(t) \Vert_{L_t^1 B_{\infty,1}^1}.$ 
\end{proposition}

\noindent We recall the following Nash-De Giorgi estimate for the  convection-diffussion equation. 
\begin{proposition}[Lemma A.1 in \cite{HR}]\label{prop:disp}

Consider the equation 
\begin{equation}\label{dis}
\partial_t f+v\cdot\nabla f-\Delta f=\nabla\cdot F+G, \quad t>0, \quad x\in\R^3,\quad 
f_{|t=0}=f_{0}.
\end{equation}
Consider $p,q,p_1,q_1\in[0,\infty]$, $r\in[2,\infty]$ with 
$
\,\frac2p+\frac3q<1, \quad \frac2{p_1}+\frac3{q_1}<2.
$
There exists $C >0$ such that for every smooth divergence-free vector field $v$ and every $T >0$, if
$F\in L^p_TL^q$, $G\in L^{p_1}_TL^{q_1}$ and $f_0\in L^r$ then the solution of \eqref{dis} satisfies the estimate:
\begin{align*}
\Vert f(T)\Vert_{L^\infty}\leq &C(1+T^{-\frac{3}{2r}})\Vert f_0\Vert_{L^r}+C\big(1+\sqrt{T}^{1-(\frac2p+\frac3q)}\big)\Vert F\Vert_{L^p_TL^q}+C\big(1+\sqrt{T}^{2-(\frac2{p_1}+\frac3{q_1})}\big)\Vert G\Vert_{L^{p_1}_TL^{q_1}}.
\end{align*}
\end{proposition}
\section{Axisymmetric flows}
Along this paper, we assume that    vector fields $v :\R^3\to\R^3$ and $b:\R^3\to\R^3$  are   axisymmetric,  namely,
$$
\mathcal{R}_{-\alpha}\{v(\mathcal{R}_\alpha x)\}=v(x),\quad\textnormal{and}\quad \mathcal{R}_{-\alpha}\{b(\mathcal{R}_\alpha x)\}=b(x),\quad \forall\alpha\in[0,2\pi],\quad \forall x\in\R^3,
$$
where $\mathcal{R}_\alpha$ denotes the rotation around the axis $(Oz)$  with angle $\alpha$. The vector field $v$ is assumed to be  without swirl and 
the radial and the height components of magnetic velocity field $b$ are  trivial, that is,
\be\label{axiv}
v(x) = v^r(r, z)e_r + v^z (r, z)e_z,\qquad b(x) = b^\theta(r, z)e_\theta,
\ee
where $ x=(x_{1},x_{2},z),\; r=({x_{1}^2+x_{2}^2})^{\frac12},$ and   $(e_r, e_{\theta} , e_z)$ is the cylindrical basis of
$\mathbb R^3$, 
$$
e_{r}:=(\cos\theta, \sin\theta, 0)^T, \quad e_{\theta}:=(-\sin\theta, \cos\theta, 0)^T,\quad e_{z}:=(0, 0, 1)^T.
$$
The components $v^r$,  $v^z$ and $b^\theta$ do not depend on the angular variable $\theta$. 
Recall that, in the cylindrical coordinate systems, the gradient operator $\nabla$ is given by
\be\label{nabla}
\nabla=e_r \partial_{r}+\frac{1}{r}e_\theta\partial_\theta+e_{z}\partial_{z}.
\ee
and the expression of the Laplacian operator  is given by
\be\label{lap}
\Delta=\partial_{rr}+\frac{1}{r}\partial_r+\partial_{zz}.
\ee
The incompressible constraints, div$\, v=0$ and div$\, b=0$,  can be written as
\be\label{div}
\partial_r v^r+\frac{1}{r}v^r+\partial_zv^z=0, \quad\textnormal{and}\quad \partial_r b^r+\frac{1}{r}b^r+\partial_zb^z=0.
\ee
Moreover, for any axisymmetric vector filed $a$, one has
\be\label{dotnabla}
a\cdot \nabla=a^{r}\partial_{r}+\frac{a^\theta}{r}\partial_\theta+a^{z}\partial_{z}.
\ee
 Thus, straightforward computations show that the equations \eqref{MHD} can be expressed in cylindrical coordinate as
\begin{equation}
\label{eq:mhd1}
\begin{cases}
\partial_t  v^r +  v^r \partial_r  v^r+  v^z \partial_z  v^r + \partial_r p=-\frac{(b^\theta)^2}{r},   \\
\partial_t  v^z +  v^r \partial_r  v^z+  v^z \partial_z  v^z + \partial_z p=0, \\
\partial_t  b^\theta +  v^r \partial_r  b^\theta+  v^z \partial_z  b^\theta - \Delta b^\theta=\frac{b^\theta}{r} v^r.    
\end{cases}
\end{equation}
Furthermore, one can easily check that  the vorticity $\omega:=\textnormal{curl}\, v$ of the vector filed $v$ is given by
$$
\omega=\omega^\theta e_\theta;\quad \omega^\theta=\partial_zv^r-\partial_rv^z.
$$
Now,  applying the curl operator to the velocity equation yields
\begin{equation}\label{voticityequation}
\partial_t \omega+v \cdot \nabla \omega=\omega\cdot\nabla v+\textnormal{curl}(b\cdot\nabla b).
\end{equation}
Since $\omega=\omega^\theta e_\theta$ and by \eqref{dotnabla}, the stretching term takes the form
 \begin{align}\label{stretch}
\omega\cdot\nabla v
 &=\omega^\theta\,\frac{ v^r}{r}\, e_\theta 
 = \frac{v^r}{r}\omega.
\end{align} 
Similarly, as $b=b^\theta e_\theta$ we get 
\begin{align}\label{bnablab}
b\cdot\nabla b= -\frac{(b^\theta)^2}{r}e_r\qquad\textnormal{and}\qquad\textnormal{curl}(b\cdot\nabla b)&=-\partial_z\Big(\frac{b^\theta}{r}b\Big).
\end{align}
Consequently the vorticity equation \eqref{voticityequation} becomes
\be\label{eq:vorticity}
\partial_t \omega+v \cdot \nabla \omega=\frac{v^r}{r}\omega-\partial_z\Big(\frac{b^\theta}{r}b\Big).
\ee
It follows that the scalar function  $\omega^\theta$ verifies
\begin{equation}
\label{eq:mhd2}
\partial_t  \omega^\theta +  v \cdot \nabla\omega^\theta=\frac{v^r}{r}\omega^\theta-\partial_z\frac{(b^\theta)^2}{r}.  
\end{equation} 
Moreover, from the third equation of \eqref{eq:mhd1}, one has
\begin{equation}
\label{eq:mhd22}
\partial_t  b^\theta +  v \cdot \nabla  b^\theta - \Delta b^\theta=\frac{b^\theta}{r} v^r .
\end{equation}
Therefore the quantities 
$
\Omega:=\frac{\omega^\theta}{r}$ and  $\Gamma:=\frac{b^\theta}{r}
$
are  governed by the following system
\begin{equation}
\label{eq:mhd3}
\begin{cases}
\partial_t  \Omega +  v \cdot \nabla\Omega=-\partial_z\Gamma^2,   \\
\partial_t  \Gamma + v \cdot \nabla \Gamma - \Big(\Delta +\frac{2}{r}\partial_r\Big)\Gamma =0 . 
\end{cases}
\end{equation}

\noindent The following proposition is proved in \cite{AHK, H}.
\begin{proposition}\label{prop:geo}
Let $v=(v^{1}, v^{2}, v^{3})$ and $b=(b^{1}, b^{2}, b^{3})$ be a smooth axisymmetric vector fields as in \eqref{axiv}. Then we have
\begin{enumerate}[label=\rm(\roman*)]
\item For every  $(x_{1}, x_{2}, z)\in \R^{3}$, 
\begin{align*}
v^{z}=0,\quad -x_{2}v^{1}(x_{1}, x_{2}, z)+x_1 v^{2}(x_{1}, x_{2}, z)&=0,\\
b^{z}=0,\quad \;\;\, x_{1}b^{1}(x_{1}, x_{2}, z)+x_2\, b^{2}(x_{1}, x_{2}, z)&=0,\\
v^{1}(0, x_{2}, z)=v^{2}(x_{1},0, z)=b^{1}(x_{1}, 0, z)=b^{2}(0, x_{2}, z)&= 0.
\end{align*}
\item For every $q\geq -1$, $\Delta_{q}v$ and $\Delta_{q}b$ are axisymmetric and 
$$
(\Delta_{q}v^{1})(0, x_{2}, z)=(\Delta_{q}v^{2})(x_1, 0, z)=0,
$$ 
$$
(\Delta_{q}b^{1})(x_1, 0, z)=(\Delta_{q}b^{2})(0, x_2, z)=0.
$$ 
\item The vector  $\omega=\nabla\times v=(\omega^{1}, \omega^{2}, \omega^{3})$ satisfies $\omega=\nabla\times e_{\theta}=(0, 0, 0)$and we have for every  $(x_{1}, x_{2}, z)\in \R^{3}$, 
$$
\omega^{z}=0,\quad x_{1}\omega^{1}(x_{1}, x_{2}, z)+x_{2}\omega^{2}(x_{1}, x_{2}, z)=0,\quad \omega^{1}(x_{1}, 0, z)=\omega^{2}(0, x_{2}, z)= 0.$$
\end{enumerate}   
\end{proposition}
Now, we aim to study the following  vorticity like equation,  in which no relations between the vector field $v$ and the solution $\zeta$ are supposed 
\begin{equation} \label{m2}
\left\{ \begin{array}{lll}
\partial_{t}\zeta+(v\cdot\nabla)\zeta=(\zeta\cdot\nabla)v+\textnormal{curl}(b\cdot\nabla b),\\
\textnormal{div}v=0,\\
\zeta_{\mid t=0}=\zeta_{0},
\end{array} \right.
\end{equation}
where $v$ and $b$ are axisymmetric and verify \eqref{axiv}  and  $\zeta=(\zeta^{1}, \zeta^{2}, \zeta^z)$ is an unknown vector field. 
\begin{proposition}\label{prop:geo2}
Let $v$ be a smooth, divergence free, axisymmetric vector field and $\zeta$ be the unique global solution of \eqref{m2} with smooth initial data $\zeta_0$. Then we have the following properties.
 \begin{enumerate}[label=\rm(\roman*)]
\item If $\textnormal{div }\zeta_{0}=0,$ then for all $t\in\mathbb{R}_{+}$,
$
\textnormal{div }\zeta(t)=0.
$
\item If $\zeta_{0}\times \vec{e}_{\theta}= 0$ then for all $t\in\mathbb{R}_{+}$,
$
\zeta(t,x)\times \vec{e}_{\theta}=(0,0,0).
$
Consequently, 
$$\zeta_{1}(t, x_{1}, 0, z)=\zeta_{2}(t, 0, x_{2}, z)=0\quad  and
\quad
(\zeta\cdot\nabla)v=\frac{v^{r}}{r}\zeta.
$$
 \end{enumerate}
\end{proposition}
\begin{proof}
 We apply the divergence operator to the equation \eqref{m2}, we get  
$$
\partial_t(\textnormal{div}\zeta)+\textnormal{div}(v\cdot\nabla \zeta)= \textnormal{div}(\zeta\cdot\nabla v),
$$
where we have used the fact that $ \textnormal{div}(\textnormal{curl}\, a)=0$  for any vector $a\in\R^3$. Straightforward computations allows to get 
\begin{eqnarray*}
\textnormal{div}(v\cdot\nabla \zeta) - \textnormal{div}(\zeta\cdot\nabla v)=v\cdot\nabla\textnormal{div}\zeta.
\end{eqnarray*}
Thus we obtain the equation
\begin{equation*}
\partial_{t}\textnormal{div}\zeta+(v\cdot\nabla)\textnormal{div}\zeta=0.
\end{equation*}
Then, by the maximum principle we conclude  (i). 

\smallskip

Next, we shall prove (ii).  Denote by $(\zeta^{r},\zeta^\theta,\zeta^z)$ the coordinates of  $\zeta$ in cylindrical basis.  Then
$$
\zeta\times\vec{e}_{\theta}=\zeta^r {e}_{z}-\zeta^z {e}_{r}
$$
Hence it suffices to prove that $\zeta^r=\zeta^z=0$. Taking the scalar product of \eqref{m2} with $e_r$ we get
$$
\partial_t \zeta^r+ (v \cdot \nabla \zeta). e_r=(\zeta\cdot\nabla v). e_r+\textnormal{curl}(b\cdot\nabla b)\cdot  e_r.
$$
From \eqref{dotnabla}  and  since $\partial_{r}{e}_{r}=\partial_{z}{e}_{r}=0$, we get
\begin{eqnarray*}
(v\cdot\nabla\zeta).{e_{r}}&=&( v^{r}\partial_{r}\zeta+v^{3}\partial_{z}\zeta)\cdot {e_{r}}\\ &=& (v^{r}\partial_{r}+v^{3}\partial_{z})(\zeta \cdot  {e_{r}})\\ &=& v\cdot\nabla \zeta^{r},
\end{eqnarray*}
Similarly we obtain
\begin{eqnarray*}
(\zeta\cdot\nabla v).{e_{r}}&=& \big(\zeta^{r}\partial_{r}v+\frac{1}{r}\zeta^{\theta}\partial_{z}v+\zeta^z\partial_{3}v\big). {e}_{r}= \zeta^{r}\partial_{r}+\zeta^z\partial_{z}v^{r}.
\end{eqnarray*}
On the other hand, by \eqref{axiv} and  \eqref{bnablab} we have
\begin{eqnarray*}
\textnormal{curl}(b\cdot\nabla b). e_r&=0.
\end{eqnarray*}
Thus $\zeta^{r}$ verifies the equation 
\begin{equation*}
\partial_{t}\zeta^{r}+v\cdot\nabla\zeta^{r}=\zeta^{r}\partial_{r}v^{r}+\zeta^z\partial_{z}v^{r}.
\end{equation*}
By the maximum principle we get
\begin{eqnarray*}
\Vert\zeta^{r}(t)\Vert_{L^{\infty}}
&\leq &\int_{0}^{t}\big(\Vert\zeta^{r}(\tau)\Vert_{L^{\infty}}+\Vert\zeta^z(\tau)\Vert_{L^{\infty}}\big)\Vert\nabla v(\tau)\Vert_{L^{\infty}}d\tau.
\end{eqnarray*}
In a similar way  we may check that
\begin{equation*}
\partial_{t}\zeta^z+v\cdot\nabla\zeta^z=\zeta^{r}\partial_{r}v^{3}+\zeta^z\partial_{z}v^{3}.
\end{equation*}
Therefore, we have
\begin{equation*}
\Vert\zeta^z(t)\Vert_{L^{\infty}}\leq\int_{0}^{t}\big(\Vert\zeta^{r}(\tau)\Vert_{L^{\infty}}+\Vert\zeta^z(\tau)\Vert_{L^{\infty}}\big)\Vert\nabla v(\tau)\Vert_{L^{\infty}}d\tau.
\end{equation*}
Combining the two last estimates with  Gronwall’s inequality, we obtain for every $t\in \R+$,
$$
\zeta^{r}(t)=\zeta^z(t)=0.
$$
It follows that
$
\zeta\times {e}_{\theta}=0
$
and the the stretching term becomes
\begin{eqnarray*}
\zeta\cdot\nabla v
 &=&\zeta^\theta\,\frac{ v^r}{r}\, e_\theta 
 =\frac{v^r}{r}\zeta.
\end{eqnarray*} 
The proof of the Proposition is now complete.\qed
\end{proof}

\section{A priori estimates}
\quad In this section we shall establish some global a priori estimates needed for the proof of Theorem \ref{thm:sub-critical} and Theorem \ref{thm:critical}. First,  we shall  prove some basic weak estimates that can be easily obtained  through energy type estimates. In a second step, we shall give a global estimate on  $\|{\omega}/{r}\|_{L_t^\infty L^{3,1}}$ and finally we intend to show the control of some stronger norms such as $\|{\omega}\|_{L_t^\infty L^\infty}$ and $\|b \|_{L_t^1 B^0_{\infty,1}}$. 

\smallskip

Recall that we always denote
$$\Phi_{l}(t)=C_{0}\underbrace{\exp(...\exp}_{l-times}(C_{0} t^{\frac54})...),$$
where $C_0$ depends on the involved norms of the initial data and its value may vary from line to line
up to some absolute constants.

\smallskip

We start with the following elementary  estimates.
\begin{proposition}\label{prop:energy}
Let $(v,b)$ be a smooth solution of \eqref{MHD} then 
\begin{enumerate}[label=\rm(\roman*)]
\item For $v_0,b_0\in L^2$ and $t\in\R_+$  we have  the energy law of the three-dimensional incompressible MHD equations
\begin{align}\label{energ-est}
\Vert v\Vert_{L_t^\infty L^2}^2+\Vert b\Vert_{L_t^\infty L^2}^2+2\Vert \nabla b\Vert_{L_t^2 L^2}^2= \Vert v_0\Vert_{L^2}^2+\Vert b_0\Vert_{L^2}^2.
\end{align}
\item Assume that  $(v,b)$ is an axisymmetric solution of  \eqref{MHD}. Then, for all $t\in\R_+$  we have
\begin{align}\label{lpgamma2}
\Big\Vert \frac{b^\theta}{r}\Big\Vert_{L_t^\infty L^2}^2+\Big\Vert\nabla \Big(\frac{b^\theta}{r}\Big)\Big\Vert_{L_t^2L^2}^2\leq \Big\Vert \frac{b^\theta_0}{r}\Big\Vert_{L^2}^2. 
\end{align}
In addition,   for all $p\in(1,\infty],\, q\in [1,\infty]$ and  $t\in\R_+$,
\begin{align}\label{lpgamma}
 \Big\Vert \frac{b^\theta}{r}\Big\Vert_{L_t^\infty L^{p,q}}\leq \Big\Vert \frac{b^\theta_0}{r}\Big\Vert_{L^{p,q}}.
\end{align}
\item  Assume that  $(v,b)$ is an axisymmetric solution of  \eqref{MHD} such that  $v_0,b_0\in L^2$,  and  ${b^\theta_0}/{r}\in L^m$, $m>6$. Then  for all  $t\in\R_+$ we have 
\begin{equation}\label{boundbp}
\Vert b(t)\Vert_{L^\infty}\leq  C_0\big(t^{-\frac34}+t^{\frac14}\big), 
\end{equation}
and
\begin{equation}\label{l1tb}
\Vert b\Vert_{L_t^1L^p}\leq C_0(t^\frac14+t^\frac54) \quad  \textnormal{for all} \quad p\in(2,\infty].
\end{equation}

\end{enumerate}

\end{proposition}
\begin{proof}
 By taking the $L^2$-scalar product of the first equation of \eqref{MHD} with $v$ and integrating
by parts we get, since $v$ and $b$ are divergence free, 
$$
\frac12\frac{d}{dt}\Vert v(t)\Vert_{L^2}^2=-\int_{\R^3}(b\cdot\nabla v)\cdot b \,dx.
$$
Making a similar energy estimate on the second  equation of \eqref{MHD}  yields
$$
\frac12\frac{d}{dt}\Vert b(t)\Vert_{L^2}^2+\int_{\mathbb{R}^3}|\nabla b(t,x)|^2dx=\int_{\R^3}(b\cdot\nabla v)\cdot  b \, dx.
$$
Summing the two last identities gives
$$
\frac12\frac{d}{dt}\Vert v(t)\Vert_{L^2}^2+\frac12\frac{d}{dt}\Vert b(t)\Vert_{L^2}^2+\int_{\mathbb{R}^3}|\nabla b(t,x)|^2dx=0.
$$
Integrating in time gives the desired result in (i).

\smallskip

Recall that the evolution of the quantity  $\Gamma:=b^\theta/r$ is given by
\begin{equation*}
\partial_t  \Gamma + v \cdot \nabla \Gamma - \Big(\Delta +\frac{2}{r}\partial_r\Big)\Gamma=0.
\end{equation*}
Multiplying the last equation  by $\Gamma$ an integrating in space  we get
\begin{align*}
&\frac12\frac{d}{dt}\Vert \Gamma(t)\Vert_{L^2}^2+\Vert\nabla \Gamma(t)\Vert_{L^2}^2+2\pi\int_{\mathbb{R}}|\Gamma(t,0,z)|dz=0.
\end{align*}
 Similarly,  
  for all $p\in[1,\infty]$ we have
\be\label{normgammalp}
\Vert\Gamma(t)\Vert_{L^p}\leq \Vert\Gamma_0\Vert_{L^p}.
\ee 
By interpolation , for  $p\in ]1,\infty[$ and $q\in [1,\infty]$, we conclude (ii).

\smallskip

Now we shall prove (iii). Since $b=b^\theta e_\theta$ then  from \eqref{eq:mhd22}   we have
\begin{equation}
\label{eq:onb}
\partial_t  b +  v \cdot \nabla  b - \Delta b=\frac{b}{r} v^r .
\end{equation} 
Applying Proposition \ref{prop:disp} to   \eqref{eq:onb} with $F=0$ and $G=\frac{b}{r}v_r\,$ we get, for all $m>6$,
\begin{align*}
\Vert b(t)\Vert_{L^\infty}&\leq C(1+t^{-\frac{3}{4}})\Vert b_0\Vert_{L^{2}}+C(1+t^{(\frac14-\frac{3}{2m})})\Big\Vert v^r\, \frac{b}{r}\Big\Vert_{L^{\infty}_tL^{\frac{2m}{m+2}}}
\\ & \leq C(1+t^{-\frac{3}{4}})\Vert b_0\Vert_{L^{2}}+C(1+t^{(\frac14-\frac{3}{2m})})\Vert v\Vert_{L^{\infty}_tL^{2}}\Big\Vert \frac{b^\theta}{r}\Big\Vert_{L^{\infty}_tL^{m}}.
\end{align*}
Then, by \eqref{energ-est} and \eqref{lpgamma} we conclude that
\begin{equation}\label{boundinfty}
\Vert b(t)\Vert_{L^\infty}\leq  C\big(\Vert b_0\Vert_{L^2}+\Vert v_0\Vert_{L^2}\big)\Big(1+\Big\Vert \frac{b^\theta_0}{r}\Big\Vert_{L^m}\Big)\big(t^{-\frac34}+t^{\frac14}\big).
\end{equation}  
 The estimate \eqref{l1tb}   can be immediately obtained by using an interpolation inequality in the $L^p$ spaces, with $p\in (2,\infty)$, the dispersive estimate \eqref{boundinfty},  the energy estimate \eqref{energ-est} and finally integration in time. 
 \qed
\end{proof}

\vspace{0.2cm}

The following Proposition provides the main important quantities that one should estimate in order to get the global existence of smooth solutions.

\begin{proposition}\label{prop:gamma}
Let $v_0, b_0\in L^2$ be  two smooth axisymmetric vector fields with zero divergence such that  ${\omega_0^\theta }/{r}\in L^{3,1}$ and    ${b_0^\theta}/{r}\in L^2\cap L^{\infty}$. Then we have for every   $t\in\mathbb{R}_+$,
\begin{equation*}
\Big\Vert \frac{v^r}{r}\Big\Vert_{L^\infty_tL^\infty}+\Big\Vert \frac{\omega}{r}\Big\Vert_{L^\infty_tL^{3,1}}+\Vert \Gamma\Vert_{L^1_t B_{2,1}^{\frac32}}\leq \Phi_1(t).
\end{equation*}
\end{proposition}
\begin{proof}
 We use the following inequality due to  Shirota and  Yanagisawa \cite{SY}. See also \cite{AHK}.
$$
\Big|\frac{v^r}{r}\Big|\lesssim \frac{1}{|.|^2}*\Big|\frac{\omega^\theta}{r}\Big|.
$$
Since $\frac{1}{|.|^2}\in L^{\frac32,\infty}$ then by the convolution law $L^{p,q}*L^{p',q'}\to L^\infty$ we get
\begin{align}\label{vr/r}
\Big\Vert \frac{v^r}{r}\Big\Vert_{L^\infty}&\lesssim  \Big\Vert\frac{\omega^\theta}{r}\Big\Vert_{L^{3,1}} .
\end{align}
On the other hand, we recall that the evolution of the quantity  $\Omega:=\omega^\theta/r$ is given by
\begin{equation*}
\partial_t  \Omega + v \cdot \nabla \Omega= -\partial_z\Gamma^2.
\end{equation*}
with $\Gamma:=b^\theta/r$.
By making an  $L^{p}$ estimates and interpolating we get 
\begin{align*}
\nonumber\Vert \Omega(t)\Vert_{L^{3,1}}&\leq \Vert \Omega_0\Vert_{L^{3,1}}+ \int_0^t\Vert \partial_z\Gamma^2(\tau)\Vert_{L^{3,1}}d\tau
\\
&\leq \Vert \Omega_0\Vert_{L^{3,1}}+ 2\int_0^t\Vert \Gamma(\tau)\Vert_{L^{\infty}}\Vert \partial_z\Gamma(\tau)\Vert_{L^{3,1}}d\tau,
\end{align*}
where we have used in  the last estimate the inequality \eqref{lor}.  
Then using the embedding  \eqref{emlb}  and Proposition \ref{prop:energy}-(ii)   we find
\begin{align}\label{omegl2lp}
\Vert \Omega(t)\Vert_{L^{3,1}}&\leq \Vert \Omega_0\Vert_{L^{3,1}}+ 2\Vert \Gamma_0\Vert_{ L^{\infty}}\Vert \Gamma\Vert_{L_t^1B_{2,1}^{\frac32}}.
\end{align}
Now, we shall estimate $\Vert \Gamma\Vert_{L^1_t B_{2,1}^{\frac32}}$. Let $q\in \mathbb{T}$ and set $\Gamma_q:=\Delta_q\Gamma$. Then, localizing in frequency the second  equation of \eqref{eq:mhd3} gives
\begin{equation}\label{eq:gammaq}
\partial_t  \Gamma_q + v \cdot \nabla \Gamma_q - \Big(\Delta +\frac{2}{r}\partial_r\Big)\Gamma_q=-[\Delta_q,v\cdot\nabla]\Gamma.
\end{equation}
Multiplying \eqref{eq:gammaq} by $|\Gamma_q|^{{p}-2}\Gamma_q$, $p\geq 2$, and using H\"older inequality, we get
$$
\frac{1}{{p}}\frac{d}{dt}\Vert\Gamma_q(t)\Vert_{L^{{p}}}^{{p}}-\int_{\mathbb{R}^3}(\Delta \Gamma_q)|\Gamma_q|^{{p}-2}\Gamma_qdx-\int_{\mathbb{R}^3}\frac{2}{r}\partial_r\Gamma_q|\Gamma_q|^{{p}-2}\Gamma_qdx\leq \Vert\Gamma_q\Vert_{L^{{p}}}^{{p}-1}\Vert[\Delta_q,v\cdot\nabla]\Gamma\Vert_{L^{{p}}}.
$$
Obviously,
$$
0 \leq-\int_{\mathbb{R}^3}\frac{2}{r}\partial_r\Gamma_q|\Gamma_q|^{{p}-2}\Gamma_qdx.
$$
Moreover, using the generalized Bernstein inequality, see \cite{L}, we find
$$
\frac{1}{{p}} 2^{2q}\Vert\Gamma_q(t)\Vert_{L^{{p}}}^{{p}}\leq-\int_{\mathbb{R}^3}(\Delta \Gamma_q)|\Gamma_q|^{{p}-2}\Gamma_qdx.
$$ 
Combining the three last estimates gives 
\begin{equation*}
\frac{d}{dt}\Vert\Gamma_q(t)\Vert_{L^{{p}}}+c_{{p}} 2^{2q}\Vert\Gamma_q\Vert_{L^{{p}}}\lesssim \Vert[\Delta_q,v\cdot\nabla]\Gamma\Vert_{L^{{p}}}.
\end{equation*}
This leads to
\begin{equation*}
\Vert\Gamma_q(t)\Vert_{L^{{p}}}\leq e^{-c_{{p}}t  2^{2q}}\Vert\Delta_q \Gamma_0\Vert_{L^{{p}}}+\int_0^t e^{-c_p 2^{2q}(t-\tau)}\Vert[\Delta_q,v\cdot\nabla]\Gamma\Vert_{L^{{p}}}d\tau.
\end{equation*}
Therefore,  by integration in time,  we obtain 
\begin{equation}\label{deltagama}
\Vert\Gamma_q\Vert_{L^\infty_t L^{{p}}}+2^{2q}\Vert\Gamma_q\Vert_{L^1_tL^{{p}}}\lesssim \Vert\Gamma_0\Vert_{L^{{p}}}+\Vert[\Delta_q,v\cdot\nabla]\Gamma\Vert_{L^1_tL^{{p}}}.
\end{equation}
In particular, for $p=2$, we have 
\begin{equation*}
\Vert\Gamma_q\Vert_{L^1_t L^{2}}\lesssim 2^{-2q}\Vert\Gamma_0\Vert_{L^2}+2^{-2q}\Vert[\Delta_q,v\cdot\nabla]\Gamma\Vert_{L^1_tL^2}.
\end{equation*}
Multiplying this last inequality by $2^{q\frac32}$, then  summing over $q$ we find 
\begin{align*}
\nonumber\Vert \Gamma\Vert_{L^1_t B_{2,1}^{\frac32}}
 &\lesssim \Vert \Delta_{-1}\Gamma\Vert_{L^1_t L^2}+\sum_{q\geq 0}2^{-q\frac12}\Vert \Gamma_0\Vert_{ L^2}+\sum_{q\geq 0}2^{-q\frac12}\Vert  [\Delta_q,v\cdot\nabla]\Gamma\Vert_{L^1_tL^2}.
\end{align*}
Then, we shall use   Proposition \ref{prop:commu2} and  Proposition \ref{prop:energy}-(ii)  to obtain 
\begin{align}\label{nablagamma2}
\nonumber\Vert \Gamma\Vert_{L^1_t B_{2,1}^{\frac32}}& \lesssim \Vert \Gamma_0\Vert_{ L^2}+\Vert \Gamma\Vert_{L^1_t L^2}+\int_0^t\Vert \Omega(\tau)\Vert_{L^{3,1}}\big(\Vert b^\theta(\tau)\Vert_{L^6}+\Vert \Gamma(\tau) \Vert_{L^2}\big)d\tau\\ & \lesssim \Vert \Gamma_0\Vert_{ L^2}(t+1)+\int_0^t\Vert \Omega(\tau)\Vert_{L^{3,1}}\big(\Vert b^\theta(\tau)\Vert_{L^6}+\Vert \Gamma_0 \Vert_{L^2}\big)d\tau.
\end{align}
Inserting the last estimate into \eqref{omegl2lp} gives
\begin{align*}
\Vert \Omega(t)\Vert_{L^{3,1}}&\leq C_0(t+1)+C_0\int_0^t\Vert \Omega(\tau)\Vert_{L^{3,1}}\big(\Vert b^\theta(\tau)\Vert_{L^6}+1\big)d\tau.
\end{align*}
The Gronwall lemma and Proposition \ref{prop:energy}-(iii) yield that
\begin{align}\label{om2p}
\Vert \Omega(t)\Vert_{L^{3,1}}&\leq C_0\big(t+1)e^{C_0\Vert b^\theta\Vert_{L_t^1L^6}}\leq C_0e^{C_0 t^\frac54 }.
\end{align}
This  gives in turn, by \eqref{nablagamma2} and Proposition \ref{prop:energy}-(iii),
\begin{align*}
\Vert \Gamma\Vert_{L^1_t B_{2,1}^{\frac32}} &\leq  C_0 (t+1)+C_0(1+t^\frac54) e^{C_0 t^\frac54 }\leq \Phi_1(t).
\end{align*}
Moreover, by   \eqref{vr/r}, we get
\begin{align*}
\Big\|\frac{v^r}{r}(t)\Big\|_{L^\infty}\lesssim \Vert \Omega(t)\Vert_{L^{3,1}}&\leq \Phi_1(t).
\end{align*}
Finally, using the fact that $\omega=\omega^\theta e_\theta$ and from \eqref{lor} we infer that
$$
\Big\Vert \frac{\omega}{r}(t)\Big\Vert_{L^{3,1}}\lesssim \Vert \Omega(t)\Vert_{L^{3,1}}\leq \Phi_1(t).
$$
This concludes the proof of the proposition.\qed
 \end{proof}

 \vspace{0,2cm}
 
 With the estimates established in Proposition \ref{prop:energy} and Proposition \ref{prop:gamma} in hand, we can bound $\Vert \omega\Vert_{L^\infty_tL^{\infty}}$ by using the smoothing effect on the magnetic field.
 
\begin{proposition}\label{propo:bound-vort}
Let $\sigma\in(3,\infty)$, $v_0\in L^2$ and $b_0\in L^2\cap B_{{\sigma},1}^{{3}/{\sigma}-1}$ be  two smooth axisymmetric vector fields with zero divergence   such that $\omega_0\in L^\infty$, ${\omega_0^\theta}/{r}\in L^{3,1}$ and    ${b_0^\theta}/{r}\in L^2\cap L^{\infty}$.  Then we have for every $t\in\mathbb{R}+$,
\begin{align*}
&
\Vert b\Vert_{{L}^\infty_t B_{{\sigma},1}^{{3}/{\sigma}-1}}+\Vert b\Vert_{{L}^1_t B_{{\sigma},1}^{{3}/{\sigma}+1}} +\Vert \omega\Vert_{L^\infty_tL^{\infty}}
+\Vert b\Vert_{L^2_t B_{\infty,1}^0} 
\leq \Phi_2(t).
\end{align*}

\end{proposition}
\begin{proof}
Applying the maximum principle to \eqref{eq:vorticity}  gives
\begin{align}\label{estomlinf}
\Vert \omega(t)\Vert_{L^{\infty}}&\lesssim \Vert \omega_0\Vert_{L^{\infty}}+\int_0^t\Vert \partial_z b(\tau)\Vert_{L^\infty}\Vert b^\theta(\tau)/r\Vert_{L^{\infty}}d\tau+\int_0^t\Vert v^r(\tau)/r\Vert_{L^\infty}\Vert \omega(\tau)\Vert_{L^{\infty}}d\tau\nonumber \\&\lesssim \Vert \omega_0\Vert_{L^{\infty}}+\Vert b_0^\theta/r\Vert_{L^{\infty}}\Vert b\Vert_{L_t^1B_{{\sigma},1}^{\frac{3}{\sigma}+1}}+\int_0^t\Vert v^r(\tau)/r\Vert_{L^\infty}\Vert \omega(\tau)\Vert_{L^{\infty}}d\tau
\end{align}
where we have used in the last estimate   the embedding $B_{{\sigma},1}^{\frac{3}{\sigma}+1}\hookrightarrow\textnormal{Lip}(\R^3) $ and  Proposition \ref{prop:energy}-(ii). Now we shall estimate $\Vert b\Vert_{L_t^1B_{{\sigma},1}^{\frac{3}{\sigma}+1}}$.
Localizing in frequency the equation \eqref{eq:onb} gives
$$
\partial_t b_q+v\cdot\nabla b_q-\Delta b_q=F_q-[\Delta_q,v \cdot \nabla]b
$$
for all  $q\in\mathbb{N}\cup\{-1\}$ where $b_q:=\Delta_q b$ and  $F_q :=\Delta_q \big(b\cdot\nabla v\big)=\Delta_q \big(\frac{v^r}{r}b\big)$. Then taking the standard $L^p$-estimate yields, for all $p\geq 2$,
$$
\frac{1}{p}\frac{d}{dt}\Vert b_q(t)\Vert_{L^{p}}^{p}-\int_{\mathbb{R}^3}(\Delta  b_q)|  b_q |^{p-2} b_q dx\leq  \Vert b_q(t)\Vert_{L^{p}}^{p-1}\Big(\Vert F_q(t)\Vert_{L^{p}}+\Vert [\Delta_q, v \cdot \nabla] b(\tau) \Vert_{L^p}\Big).
$$
Using the generalized Bernestein inequality, see \cite{L},
$$
\frac{1}{{p}} 2^{2q}\Vert b_q(t)\Vert_{L^{{p}}}^{{p}}\leq-\int_{\mathbb{R}^3}(\Delta b_q)|b_q|^{{p}-2}b_qdx,
$$
we obtain
$$
\frac{d}{dt}\Vert b_q(t)\Vert_{L^{p}}+c_p 2^{2q}\Vert b_q(t)\Vert_{L^{p}}\lesssim  \Vert F_q(t)\Vert_{L^{p}}+\Vert [\Delta_q, v \cdot \nabla] b(t) \Vert_{L^p}.
$$
This leads to
\begin{equation*}
\Vert b_q(t)\Vert_{L^{p}}\lesssim e^{-c_pt2^{2q}}\Vert b_q(0)\Vert_{L^{p}}+\int_0^te^{-c(t-\tau)2^{2q}}\Big(\Vert F_q(\tau)\Vert_{L^{p}}+\Vert [\Delta_q, v \cdot \nabla] b(\tau) \Vert_{L^p}\Big)d\tau.
\end{equation*}
By Young's convolution inequality we get,  for all $a\in[1,\infty)$,
\begin{equation}\label{decoup-b0}
\Vert b_q\Vert_{L^\infty_t L^p}+2^{\frac{2q}{a}}\Vert b_q\Vert_{L^a_t L^p}\leq c_a \Vert b_q(0)\Vert_{L^p}+c_a\Vert[\Delta_q,v\cdot\nabla]b\Vert_{L^1_t L^p}+\Vert F_q\Vert_{L^1_t L^{p}}.
\end{equation}
To estimate the commutator in the right hand side, we can use  Proposition \ref{prop:commu},
\begin{align*}
\Vert[\Delta_q,v\cdot\nabla]b\Vert_{L_t^1L^p}&\lesssim \int_0^t\Vert b(\tau)\Vert_{L^p}\big((q+2)\| \omega(\tau)\|_{ L^\infty}+\| v(\tau)\|_{ L^2}\big)d\tau.
\end{align*}
As for the second term in \eqref{decoup-b0}, we us the continuity of the operator $\Delta_q$ in the $L^p$ space  to write
$$
\Vert F_q\Vert_{L^1_t L^{p}}=\Big\Vert \Delta_q\Big(\frac{v^r}{r}b\Big)\Big\Vert_{L^1_t L^{p}}\lesssim \int_0^t\Big\Vert \frac{v^r}{r}(\tau)\Big\Vert_{L^{\infty}}\| b(\tau)\|_{  L^p}d\tau .
$$
Inserting the two last estimates into \eqref{decoup-b0} we find
\begin{align}\label{estim-smoothing-b}
\nonumber&\Vert b_q\Vert_{L^\infty_t L^p}+2^{\frac{2q}{a}}\Vert b_q\Vert_{L^a_t L^p}\lesssim\\ &\qquad\qquad \Vert b_q(0)\Vert_{L^p}+\Vert b\Vert_{L^1_tL^{\sigma}}\big(\| v\|_{ L^\infty_t L^2}+\Vert {v^r}/{r}\Vert_{L^\infty_t L^\infty} \big)+(q+2)\int_0^t\Vert b(\tau)\Vert_{L^p}\| \omega(\tau)\|_{ L^\infty}d\tau .
\end{align}
Using the last estimate,  with $a=1$ and $p=\sigma>3$, we obtain 
 \begin{align}\label{claim}
\nonumber\Vert b\Vert_{{L}^\infty_t B_{{\sigma},1}^{\frac{3}{\sigma}-1}}+\Vert b\Vert_{{L}^1_t B_{{\sigma},1}^{\frac{3}{\sigma}+1}} 
&\nonumber \lesssim  \|\Delta_{-1} b\|_{L^1_t L^{\sigma}}+\Vert b_0\Vert_{B_{{\sigma},1}^{\frac{3}{\sigma}-1}}+\Vert b\Vert_{L^1_tL^{\sigma}}\big(\| v\|_{ L^\infty_t L^2}+\Vert {v^r}/{r}\Vert_{L^\infty_t L^\infty} \big)\Big(\sum_{q\geq 0}2^{q(\frac{3}{\sigma}-1)}\Big)
\\ &\nonumber \quad +\int_0^t\Vert b(\tau)\Vert_{L^{\sigma}}\| \omega(\tau)\|_{ L^\infty}d\tau\Big(\sum_{q\geq 0}2^{q(\frac{3}{\sigma}-1)}(q+2)\Big)
\\ & \nonumber  \lesssim \Vert b_0\Vert_{B_{{\sigma},1}^{\frac{3}{\sigma}-1}}+\Vert b\Vert_{L^1_tL^{\sigma}}\big(1+\| v\|_{ L^\infty_t L^2}+\Vert {v^r}/{r}\Vert_{L^\infty_t L^\infty} \big)\\ &\quad +\int_0^t\Vert b(\tau)\Vert_{L^{\sigma}}\| \omega(\tau)\|_{ L^\infty}d\tau.
\end{align}
Then combining \eqref{estomlinf} and  \eqref{claim} gives
\begin{align*}
\Vert \omega(t)\Vert_{L^{\infty}}&\leq C_0+C_0 \Vert b\Vert_{L^1_tL^{\sigma}}\big(1+\| v\|_{ L^\infty_t L^2}+\Vert {v^r}/{r}\Vert_{L^\infty_t L^\infty} \big)\\ &+C_0\int_0^t \big( \Vert b(\tau)\Vert_{L^{\sigma}}+\Vert {v^r(\tau)}/{r}\Vert_{ L^\infty} \big)\Vert \omega(\tau)\Vert_{L^{\infty}}d\tau.
\end{align*}
Thanks to  Gronwall’s inequality we infer  that
\begin{align*}
\Vert \omega(t)\Vert_{L^{\infty}}&\leq C_0\Big(1+ \Vert b\Vert_{L^1_tL^{\sigma}}\big(1+\| v\|_{ L^\infty_t L^2}+\Vert {v^r}/{r}\Vert_{L^\infty_t L^\infty} \big)\Big)e^{C_0  \Vert b\Vert_{L^1_tL^{\sigma}}+C_0\Vert {v^r}/{r}\Vert_{L^\infty_t L^\infty} }.
\end{align*}
From Proposition \ref{prop:energy} and Proposition \ref{prop:gamma} we conclude that
\begin{align}\label{omeg-inf}
\Vert \omega(t)\Vert_{L^{\infty}}&\leq C_0\big(1+( t^\frac54+t^\frac14)\Phi_1(t) \big)e^{C_0  ( t^\frac54+t^\frac14)+\Phi_1(t) }\leq \Phi_2(t) .
\end{align}
This gives in turn, by \eqref{claim},
\begin{equation}\label{claim2}
\Vert b\Vert_{{L}^\infty_t B_{{\sigma},1}^{\frac{3}{\sigma}-1}}+\Vert b\Vert_{{L}^1_t B_{{\sigma},1}^{\frac{3}{\sigma}+1}}  \leq   \Phi_2(t). 
\end{equation}
Next we shall prove the   estimate on $\Vert b\Vert_{L^2_t L^{\infty}}$. By Bernstein inequality, for $\sigma>3$, combined  the estimate \eqref{estim-smoothing-b},  at $a=2$, we get
\begin{align*}
\Vert b\Vert_{L^2_t B^0_{\infty,1}} &\le\|\Delta_{-1} b\|_{L^2_t L^{\infty}}+\sum_{q\geq 0}\| \Delta_{q} b\|_{ L^2_t L^{\infty}}
\\ \nonumber&\lesssim \| b\|_{L^2_t L^{2}}+\sum_{q\geq 0}2^{q\frac{3}{\sigma}}\| b_q\|_{ L^2_t L^{\sigma}} 
\\ &\nonumber \lesssim \| b\|_{L^2_t L^{2}}+\Vert b_0\Vert_{B_{{\sigma},1}^{\frac{3}{\sigma}-1}}
 +\Vert b\Vert_{L^1_tL^{\sigma}}\big(\| v\|_{ L^\infty_t L^2}+\Vert {v^r}/{r}\Vert_{L^\infty_t L^\infty} +\| \omega\|_{L_t^\infty L^\infty}\big)
  \sum_{q\geq 0}2^{q(\frac{3}{\sigma}-1)}(q+2)
\\ & \lesssim \| b\|_{L^2_t L^{2}}+\Vert b_0\Vert_{B_{{\sigma},1}^{\frac{3}{\sigma}-1}}+\Vert b\Vert_{L^1_tL^{\sigma}}\big(\| v\|_{ L^\infty_t L^2}+\Vert {v^r}/{r}\Vert_{L^\infty_t L^\infty} +\| \omega\|_{L_t^\infty L^\infty}\big).
\end{align*}
From Proposition \ref{prop:energy}, Proposition \ref{prop:gamma} and the estimate \eqref{omeg-inf} we conclude that
 \begin{align*}
\Vert b\Vert_{L^2_t B^0_{\infty,1}} &\leq C_0(t^\frac12+1)+C_0(t^{\frac54}+t^\frac14)\big(1+\Phi_1(t)+\Phi_2(t)\big)\le  \Phi_2(t).
\end{align*}
This completes the proof of the proposition. \qed

\end{proof}

\vspace{0.2cm}

The following  global a priori estimate will be needed for the propagation of  the critical Besov  regularities.

\begin{proposition}\label{propo:bound-vel}
Let $v_0\in L^2\cap L^\infty$, $b_0\in L^2\cap B_{{\sigma},1}^{\frac{3}{\sigma}-1}$,  $3<\sigma<\infty$ be  two smooth axisymmetric vector fields with zero divergence such that $\omega_0\in L^{\infty}$,  ${\omega_0 }/{r}\in L^{3,1}$ and    ${b_0^\theta}/{r}\in L^2\cap L^{\infty}$.
Then,  for every $t\in\mathbb{R}+$,
\begin{align*}
&\Vert v\Vert_{L_t^\infty L^{\infty}}\leq \Phi_3(t).
\end{align*}
\end{proposition}
\begin{proof}
 We will use an argument due to Serfati \cite{Serfati} and applied for the
Euler case,
\begin{align}\label{vinf}
&\Vert v(t)\Vert_{L^{\infty}}\leq \Vert \dot{S}_{-N} v(t)\Vert_{L^{\infty}}+\sum_{q\geq -N}\Vert \Delta_q v(t)\Vert_{L^{\infty}}
\end{align}
where $N$ is an arbitrary positive integer that will be fixed later. By Bernstein inequality we get
\begin{align}\label{vinf1}
\sum_{j\geq -N}\Vert \Delta_q v(t)\Vert_{L^{\infty}}\lesssim \sum_{q\geq -N}2^{-q}\Vert \Delta_q \omega(t)\Vert_{L^{\infty}}\lesssim 2^{N}\Vert  \omega(t)\Vert_{L^{\infty}}.
\end{align}
Using  the matrix Leray operator $\mathbb{P}$, the first equation of \eqref{MHD} can be reformulated  as follows
\begin{equation}\label{veleq}
\partial_{t}v=\mathbb{P}\big((b\cdot\nabla) b-(v\cdot\nabla) v\big).
\end{equation}
Therefore, one has the estimate
\begin{align*}
\Vert \dot{S}_{-N}v\Vert_{L^{\infty}}
&\le \Vert \dot{S}_{-N}v_{0}\Vert_{L^{\infty}}+\int_{0}^{t}\Big(\Vert\dot{S}_{-N}\mathbb{P}(b\cdot\nabla) b(\tau)\Vert_{L^{\infty}}+\Vert\dot{S}_{-N}\mathbb{P}(v\cdot\nabla) v(\tau)\Vert_{L^{\infty}}\Big)d\tau
\\ &\le \Vert v_{0}\Vert_{L^{\infty}}+\sum_{q\leq -N}\int_{0}^{t}\Big(\Vert\dot{\Delta}_{q}\textnormal{div}\,\mathbb{P}(b\otimes b)(\tau)\Vert_{L^{\infty}}+\Vert\dot{\Delta}_{q}\textnormal{div}\,\mathbb{P}(v\otimes v)(\tau)\Vert_{L^{\infty}}\Big)d\tau
\\ &\lesssim \Vert v_{0}\Vert_{L^{\infty}}+\sum_{q\leq -N}2^q\int_{0}^{t}\Big(\Vert\dot{\Delta}_{q}\mathbb{P}(b\otimes b)(\tau)\Vert_{L^{\infty}}+\Vert\dot{\Delta}_{q}\mathbb{P}(v\otimes v)(\tau)\Vert_{L^{\infty}}\Big)d\tau.
\end{align*}
 Since $\dot\Delta_{q}\mathbb{P}$ is continuous from $L^\infty$  to $L^\infty$ we deduce that
\begin{equation}\label{m11}
\Vert \dot{S}_{-N}v\Vert_{L^{\infty}}\lesssim \Vert v_{0}\Vert_{L^{\infty}}+\int_{0}^{t}\Vert b(\tau)\Vert_{L^{\infty}}^2 d\tau+2^{-N}\int_{0}^{t}\Vert v(\tau)\Vert_{L^{\infty}}^2d\tau.
\end{equation}
Inserting \eqref{vinf1} and \eqref{m11} into \eqref{vinf} we obtain
\begin{align*}
&\Vert v(t)\Vert_{L^{\infty}}\lesssim \Vert v_{0}\Vert_{L^{\infty}}+\int_{0}^{t}\Vert b(\tau)\Vert_{L^{\infty}}^2 d\tau+2^{N}\Vert  \omega(t)\Vert_{L^{\infty}}+2^{-N}\int_{0}^{t}\Vert v(\tau)\Vert_{L^{\infty}}^2d\tau.
\end{align*}
We choose $N$ such that
$$
2^{2N}\approx 1+\Vert  \omega(t)\Vert_{L^{\infty}}^{-1}\int_{0}^{t}\Vert v(\tau)\Vert_{L^{\infty}}^2d\tau,
$$
we find
\begin{align*}
&\Vert v(t)\Vert_{L^{\infty}}^2\lesssim \Vert v_{0}\Vert_{L^{\infty}}^2+\Big(\int_{0}^{t}\Vert b(\tau)\Vert_{L^{\infty}}^2 d\tau\Big)^2+\Vert  \omega(t)\Vert_{L^{\infty}}^2+\Vert  \omega(t)\Vert_{L^{\infty}}\Big(\int_{0}^{t}\Vert v(\tau)\Vert_{L^{\infty}}^2d\tau\Big).
\end{align*}
According to Gronwall’s inequality we get
\begin{align*}
\Vert v(t)\Vert_{L^{\infty}}^2&\lesssim \Big(\Vert v_{0}\Vert_{L^{\infty}}^2+\Big(\int_{0}^{t}\Vert b(\tau)\Vert_{L^{\infty}}^2d\tau\Big)^2+\Vert  \omega(t)\Vert_{L^{\infty}}^2 \Big)e^{Ct\Vert  \omega\Vert_{L_t^\infty L^{\infty}}}\\ &\lesssim \Big(\Vert v_{0}\Vert_{L^{\infty}}^2+\Vert b\Vert_{L_t^2 B_{\infty,1}^0}^4+\Vert  \omega(t)\Vert_{L^{\infty}}^2 \Big)e^{Ct\Vert  \omega\Vert_{L_t^\infty L^{\infty}}}.
\end{align*}
where we have used in the last estimate the embedding $B_{\infty,1}^0\hookrightarrow L^\infty$.
Then, by Proposition \ref{propo:bound-vort} we conclude that
\begin{align*}
&\Vert v(t)\Vert_{L^{\infty}}^2\leq \Phi_3(t)
\end{align*}
which is the desired estimate.\qed
\end{proof}

\section{Sub-critical regularities}
In this section we intend to prove Theorem \ref{thm:sub-critical}. We first deal with
some a priori estimates. In a second step we we highlight the proof of   the existence part and finally  we give the uniqueness result.
\subsection{Lipschitz bound of the velocity}
We shall prove the  global propagation of the sub-critical Sobolev regularities, which is  heavily related to the control of the Lipschitz norm of the velocity. The main result of this section is the following:

\begin{proposition}\label{prop:sub-c}
 Let $v_0\in H^s$ and  $b_0\in H^{s-2}$, $s>5/2$,  be two divergence-free axisymmetric vector field as in \eqref{axiv}. Assume in addition that ${b_0^\theta}/{r}\in L^2\cap L^\infty$ Then any smooth
solution  $(u,v)$    of the MHD
system \eqref{MHD} satisfies
\begin{align*}
&\Vert v\Vert_{\widetilde{L}^\infty_t H^{s}}+\Vert b\Vert_{\widetilde{L}^\infty_t H^{s-2}}+\Vert b\Vert_{\widetilde{L}^1_t H^{s }}\leq \Phi_5(t) .
\end{align*}
\end{proposition}

\begin{proof}
We localize in frequency the equation of the velocity in \eqref{MHD}. For  $q\in\mathbb{N}\cup\{-1\}$ we set $v_q:=\Delta_q v$
and $b_q :=\Delta_q b$. Then, we have
$$
\partial_t v_q+v\cdot\nabla v_q+\nabla p_q=\Delta_q\big(b\cdot\nabla b\big)-[\Delta_q,v \cdot \nabla]v.
$$
Taking the $L^2$-inner product with $v_q$ and using the incompressibility of $v$ and $v_q$  we get
\begin{align*}
\frac{d}{dt}\Vert v_q(t)\Vert_{L^{2}}& \leq \Vert \Delta_q(b\cdot\nabla b)\Vert_{L^{2}}+\Vert [\Delta_q,v \cdot \nabla]v\Vert_{L^{2}}.
\end{align*}
Then integrating in time gives
\begin{align*}
\Vert v_q(t)\Vert_{L^{2}}& \leq \Vert v_q(0)\Vert_{L^{2}}+ \Vert\Delta_q(b\cdot\nabla b)\Vert_{{L}^1_t  L^{2}}+\Vert [\Delta_q,v \cdot \nabla]v\Vert_{{L}^1_t  L^{2}}.
\end{align*}
Multiplying this last inequality by $2^{qs}$, then taking the $\ell^2$-norm we find 
\begin{align}\label{vv}
\Vert v\Vert_{\widetilde{L}^\infty_t H^{s}}&\lesssim \Vert v_0\Vert_{H^{s}}+\Vert  b\cdot\nabla b\Vert_{\widetilde{L}^1_t H^{s}}+\big\|\big(2^{qs}\Vert [\Delta_q, v \cdot \nabla] v \Vert_{L_t^1L^2}\big)_q\big\|_{\ell^2}.
\end{align}
Next we shall make use of the following commutator estimate, see for instance Lemma B.5 in \cite{D}, 
\be\label{comest1}
\big\|\big(2^{qs}\Vert [\Delta_q, v \cdot \nabla] v \Vert_{L_t^1L^2}\big)_q\big\|_{\ell^2}\lesssim \int_0^t\|\nabla v(\tau)\|_{L^\infty}\Vert v(\tau)\Vert_{ H^{s}}d\tau.
\ee
Moreover, since $\textnormal{curl} (b\cdot\nabla b)=-\partial_z \Big(\frac{ b^\theta}{r}b\Big)$ then   by Proposition \ref{prop:app} we get
\begin{align*}
\notag\Vert b\cdot\nabla b\Vert_{ \widetilde{L}^1_t H^{s}} &\leq \Vert b\cdot\nabla b\Vert_{ \widetilde{L}^1_t L^{2}} +\Vert \textnormal{curl} (b\cdot\nabla b)\Vert_{ \widetilde{L}^1_t H^{s-1}}
\\ \notag  &\lesssim \Vert  b\Vert_{ \widetilde{L}^\infty_t L^{2}} \Vert \nabla b\Vert_{ {L}^1_tL^{\infty}}+\Big\Vert \frac{ b^\theta}{r}b\Big\Vert_{ \widetilde{L}^1_tH^{s}}\\
&\lesssim\Vert  b\Vert_{{L}^\infty_t L^{2}} \Vert \nabla b\Vert_{{L}^1_t L^{\infty}}+\Big\Vert \frac{b^\theta}{r} \Big\Vert_{{L}_t^\infty L^\infty}\Vert b\Vert_{\widetilde{L}_t^1 H^{s}} .
\end{align*}
This together with \eqref{comest1} and  \eqref{vv} gives
\begin{align}\label{vv2}
\Vert v\Vert_{\widetilde{L}^\infty_t H^{s}}&\lesssim \Vert v_0\Vert_{H^{s}}+\Vert  b\Vert_{{L}^\infty_t L^{2}} \Vert \nabla b\Vert_{{L}^1_t L^{\infty}}
 +\Big\Vert \frac{b^\theta}{r}\Big\Vert_{{L}^1_t  L^\infty}\Vert b\Vert_{\widetilde{L}_t^1 H^{s}}+\int_0^t\|\nabla v(\tau)\|_{L^\infty}\Vert v(\tau)\Vert_{ H^{s}}d\tau .
\end{align}
Next, we shall estimate $\Vert b\Vert_{\widetilde{L}_t^1 H^{s}}$.  According to 
  \eqref{decoup-b0}, one has
\begin{align*}
\Vert b_q\Vert_{L_t^\infty L^2}+2^{2q}\Vert b_q\Vert_{L_t^1L^2} &\lesssim \Vert b_q(0)\Vert_{L^2}+\Vert \Delta_q ( b\cdot\nabla v)\Vert_{L_t^1L^2}+\Vert [\Delta_q, v \cdot \nabla] b(\tau) \Vert_{L_t^1L^2}\cdot
\end{align*}
It follows that
\begin{align*}
\Vert b\Vert_{\widetilde{L}^\infty_t H^{s-2}}+\Vert b\Vert_{\widetilde{L}^1_t H^{s }}&\lesssim \Vert b_0\Vert_{H^{s-2}}+\Vert \Delta_{-1} b\Vert_{L_t^1L^2}+\Vert b\cdot\nabla v\Vert_{\widetilde{L}^1_t H^{s-2 }} 
+\big\|\big(2^{q(s-2)}\Vert [\Delta_q, v \cdot \nabla] b(\tau) \Vert_{L_t^1L^2}\big)_q\big\|_{\ell^2}.
\end{align*}
Since $s-2>\frac12$ then $H^{s-2}\cap L^\infty$ is an algebra and, by Lemma \ref{le1}, we get
\begin{align*}
\Vert b\cdot\nabla v\Vert_{\widetilde{L}_t^1 H^{s-2}} &\lesssim\Vert b\cdot\nabla v\Vert_{{L}_t^1 H^{s-2}}
\lesssim  \int_{0}^t\big(\Vert b(\tau)\Vert_{ H^{s-2}}\|\nabla v(\tau)\|_{L^\infty}+\|b (\tau)\|_{L^\infty}\Vert  v(\tau)\Vert_{ H^{s-1}}\big)d\tau,
\end{align*}
Now, in order to estimate the commutator term, we shall 
use the estimate, (see for instance Lemma 2.100 in \cite{BCD})
$$
\big\|\big(2^{q(s-2)}\Vert [\Delta_q, v \cdot \nabla] b \Vert_{L_t^1L^2}\big)_q\big\|_{\ell^2}\lesssim \int_{0}^t\big(\|\nabla v(\tau)\|_{L^\infty}\Vert b(\tau)\Vert_{H^{s-2}}+\|\nabla b(\tau)\|_{L^\infty}\Vert v(\tau)\Vert_{H^{s-2}}\big)d\tau.
$$
Combining the three last inequalities and using the embedding $H^s\hookrightarrow H^{s-2 }$ we deduce that
\begin{align}\label{bb2}
\nonumber&\Vert b\Vert_{\widetilde{L}^\infty_t H^{s-2}}+\Vert b\Vert_{\widetilde{L}^1_t H^{s }}\lesssim \Vert b_0\Vert_{H^{s-2}}+\Vert b\Vert_{L_t^1L^2}
\\ &+ \int_{0}^t\Big(\|\nabla v(\tau)\|_{L^\infty}\Vert b(\tau)\Vert_{H^{s-2}}+\big(\|\nabla b(\tau)\|_{L^\infty}+\| b(\tau)\|_{L^\infty}\big)\Vert v(\tau)\Vert_{H^{s}}\Big)d\tau.
\end{align}
Putting together \eqref{vv2} and \eqref{bb2} and using Proposition \ref{prop:energy} we find
\begin{align*}
\Vert v\Vert_{\widetilde{L}^\infty_t H^{s}}+\Vert b\Vert_{\widetilde{L}^\infty_t H^{s-2}}+\Vert b\Vert_{\widetilde{L}^1_t H^{s }}&\leq C_0\big(1+t+ \Vert \nabla b\Vert_{{L}^1_t L^{\infty}}\big)
\\ &\quad+C_0 \int_{0}^t\|\nabla v(\tau)\|_{L^\infty} \Big(\Vert b(\tau)\Vert_{H^{s-2}}+\Vert v(\tau)\Vert_{H^{s}}\Big)d\tau
\\ &\quad+C_0\int_{0}^t\big(\|\nabla b(\tau)\|_{L^\infty}+\| b(\tau)\|_{L^\infty}\big)\Vert v(\tau)\Vert_{H^{s}}d\tau.
\end{align*}
Then, an application of the Gronwall inequality gives rise to
\begin{align*}
&\Vert v\Vert_{\widetilde{L}^\infty_t H^{s}}+\Vert b\Vert_{\widetilde{L}^\infty_t H^{s-2}}+\Vert b\Vert_{\widetilde{L}^1_t H^{s }}\leq C_0(t+1)e^{C_0(\|\nabla v\|_{L^1_tL^\infty} +\|\nabla b\|_{L^1_tL^\infty} +\| b\|_{L^1_tL^\infty})}.
\end{align*}
By Proposition \ref {prop:energy}-(iii), the  embeddings   
\begin{align*}
&H^{s-2}\hookrightarrow L^{\sigma} \hookrightarrow B_{{\sigma},1}^{\frac{3}{\sigma}-1},\\
&B_{{\sigma},1}^{\frac{3}{\sigma}+1}\hookrightarrow \textnormal{Lip}(\R^3),
\end{align*}
for some $\sigma>3$, and Proposition \ref{propo:bound-vort} we deduce that
\begin{align}\label{lipv}
&\Vert v\Vert_{\widetilde{L}^\infty_t H^{s}}+\Vert b\Vert_{\widetilde{L}^\infty_t H^{s-2}}+\Vert b\Vert_{\widetilde{L}^1_t H^{s }}\leq \Phi_3(t) e^{C\|\nabla v\|_{L^1_tL^\infty}}.
\end{align}
Now, to get the global bound for the Lipschitz norm of the velocity we use the classical logarithmic estimate: for $s>\frac52$,
$$
\|\nabla v\|_{L^\infty}\lesssim \| v\|_{L^2}+\|\omega\|_{L^\infty}\log\big(e+\| v\|_{H^s}\big).
$$
Therefore, by \eqref{lipv} and Proposition \ref{propo:bound-vort},
$$
\|\nabla v\|_{L^\infty}\leq \Phi_3(t)\Big(1+\int_0^t\|\nabla v(\tau)\|_{  L^\infty}\Big).
$$
By Gronwall's inequality we get
$$
\|\nabla v\|_{L^\infty}\leq \Phi_4(t).
$$
Plugging  this estimate into \eqref{lipv} we conclude that
\begin{align*}
&\Vert v\Vert_{\widetilde{L}^\infty_t H^{s}}+\Vert b\Vert_{\widetilde{L}^\infty_t H^{s-2}}+\Vert b\Vert_{\widetilde{L}^1_t H^{s }}\le \Phi_5(t) .
\end{align*}
This completes the proof of the proposition.\qed
\end{proof}
\subsection{Existence}
We will briefly outline the proof of the existence part which is
classical. We smooth out the initial data
$$v^{n}_{0}=S_{n} v^0=2^{3n}g(2^n.)\ast v_0\quad\textnormal{and}\quad b^{n}_{0}=S_{n}b^0=2^{3n}g(2^n.)\ast b_0.$$ 
where $S_n$ is the cut-off frequency defined in Section 2. Because $g$ is radial then the functions $v^{n}_{0}$ and   $b^{n}_{0}$ remain axisymmetric. Moreover, this family is uniformly bounded in the  space of initial data.
Then, by using standard arguments based on the a priori estimates given by Proposition \ref{prop:energy}, Proposition \ref{prop:gamma}, Proposition \ref{propo:bound-vort} and Proposition \ref{prop:sub-c} we can construct a unique global solution $(v_n,b_n)_{n\in \mathbb{N}}$ in the following space
$$
v_n\in \mathcal{C}(\R_+;L^2)\cap L^1(\R_+;W^{1,\infty})\quad \textnormal{and}\quad b_n\in \mathcal{C}(\R_+;H^{s-2})\cap L^1(\R_+;W^{1,\infty}).
$$
The control is uniform with respect to the parameter $n$. Therefore 
by standard compactness arguments we can show that this family $(v_n,b_n)_{n\in \mathbb{N}}$ converges to some $(v,b)$ which satisfies our initial value problem. We omit here the details and we will next focus on the uniqueness part.

\subsection{Uniqueness result} We shall  prove the uniqueness result for the system \eqref{MHD} in the following space
$$(v,b)\in \mathcal{X}:=\big(\mathcal{C}(\mathbb{R}_+,L^2)\cap L^{1}(\mathbb{R}_+;{W}^{1,\infty})\big)^2.$$
We take two solutions $(v_j,b_j)$, with $j=1,2$ for \eqref{MHD} belonging to the space $\mathcal{X}$  with  initial data $(v^0_j,b^0_j)$, $j=1,2$ and we denote
$$v=v_2-v_1\qquad\textnormal{and}\qquad b=b_2-b_1.$$
Then we find the equations
\begin{equation}\label{s111} 
\left\{ \begin{array}{ll} 
\partial_{t} v+v_{2}.\nabla v+\nabla p=- v\cdot\nabla v_{1}+b_{2}.\nabla b+b\cdot\nabla b_{1},\\ 
\partial_{t}b+v_{2}.\nabla b-\Delta b=-v\cdot\nabla b_{1}+b_{2}.\nabla v+b\cdot\nabla v_{1},\\
v_{| t=0}= v^{0}, \quad b_{| t=0}=b^{0}.  
\end{array} \right.
\end{equation}  
Taking the $L^2$ inner product of the first equation of \eqref{s111} with $v$ and  integrating by parts, $$\frac{1}{2}\frac{d}{dt}\Vert v(t)\Vert^{2}_{L^2}=- \int_{\R^3} (v\cdot\nabla v_{1}) . v \, dx-\int_{\R^3} ( b_2. \nabla v) . b \, dx-\int_{\R^3} ( b\cdot\nabla v) . b_1 \, dx.$$
Taking  the $L^2$ inner product of the second equation of \eqref{s111} with $b$  we get 
\begin{align}\label{11}
\nonumber \frac{1}{2}\frac{d}{dt}\Vert b(t)\Vert^{2}_{L^2}+\int_{\mathbb{R}^3}|\nabla  b(t,x)|.  b_q dx&=- \int_{\R^3} (v\cdot\nabla b_1) . b \, dx+\int_{\R^3} ( b_2. \nabla v) . b \, dx+\int_{\R^3} ( b\cdot\nabla v) . b_1 \, dx.
\end{align}
Summing the two last estimate and using H\"older and Young inequalities, we get 
$$
\frac{1}{2}\frac{d}{dt}\Big(\Vert v(t)\Vert^{2}_{L^2}+\Vert b(t)\Vert^{2}_{L^2}\Big)\le \Big(\Vert\nabla v_{1}(t)\Vert_{L^\infty}+\Vert\nabla\rho_{1}(t)\Vert_{L^\infty}\Big)\Big(\Vert v(t)\Vert^{2}_{L^2}+\Vert b(t)\Vert^{2}_{L^2}\Big).$$
Gronwall's inequality yields
\begin{equation}\label{8}
\Vert v(t)\Vert^{2}_{L^2}+\Vert b(t)\Vert^{2}_{L^2}\le \Big(\Vert v_0\Vert^{2}_{L^2}+\Vert b_0\Vert^{2}_{L^2}\Big) \exp\Big(2\big({\Vert\nabla v_{1}(t)\Vert_{L_{t}^{1}L^\infty}+\Vert\nabla\rho_{1}(t)\Vert_{L_{t}^{1}L^\infty}}\big)\Big).
\end{equation}
This proves the uniqueness result.

\section{Critical regularities}

This section is devoted to the proof of Theorem \ref{thm:critical}. We restrict ourselves to the a priori estimates. The existence and uniqueness parts can be
easily obtained in a classical way.

\subsection{Vorticity decomposition}

In order to get a global bound on the quantity $\Vert\omega(t)\Vert_{B^{0}_{\infty,1}}$, we shall use the approach  developed in the 3D axisymmetric  Euler case in \cite{AHK}. We 
write a  decomposition of the vorticity based on the special structure
of axisymmetric flows. This later one  is the basic tool to get the Lipschitz   bound of the velocity.
\begin{proposition}\label{prop-a4} There exists a decomposition $(\widetilde{\omega}_{j})_{j\ge-1}$ of $\omega$ such that 
\begin{enumerate}[label=\rm(\roman*)]
\item For every $t \in\R_+,$ $\omega(t,x)=\displaystyle \sum_{j \ge -1}\widetilde{\omega}_{j}(t,x).$
\item For every $t \in\R_+,$ $j\ge -1$, $\textnormal{div}\;\widetilde{\omega}_{j}(t,x)=0.$
\item For all $j\ge -1$ we have
$$\Vert\widetilde{\omega}_{j}(t)\Vert_{L^\infty}\le C \Big(\Vert\Delta_{j}\omega^{0}\Vert_{L^\infty}+2^{j}\Big\Vert \Delta_{j}\Big(\frac{b^\theta}{r}b_1\Big)\Big\Vert_{L_{t}^{1}L^\infty}\Big)\Phi_{2}(t).$$
\item For all $ k, j \ge-1$ we have
$$\Vert\Delta_{k}\widetilde{\omega}_{j}(t)\Vert_{L^\infty}\lesssim 2^{-\vert k-j \vert}e^{C V(t)}\Big(\Vert\Delta_{j}\omega^{0}\Vert_{L^\infty}+ 2^{j}\Big\Vert \Delta_{j}\Big(\frac{b^\theta}{r}b\Big)\Big\Vert_{L_{t}^{1}L^\infty}\Big),$$
with $V(t):=\Vert v(t) \Vert_{L_t^1 B_{\infty,1}^{1}}.$
\end{enumerate}
\end{proposition}
\begin{proof}
Denote by $(\widetilde{\omega}_{j})_{j\ge-1}$ the unique global solution of the  equation,
\begin{equation}\label{rt}
\left\{
\begin{array}{ll} 
\partial_{t}\widetilde{\omega}_{j}+v\cdot\nabla\widetilde{\omega}_{j}=\widetilde{\omega}_{j}.\nabla v-\partial_z\Big[\Delta_j\Big(\frac{b^\theta}{r} b\Big)\Big],\\ 
\widetilde{\omega}_{j}(t=0)=\Delta_j\omega^{0}.
\end{array} \right.
\end{equation} 
Since $\textnormal{div}(\Delta_{j} \omega^0)=0$ then from Proposition \ref{prop:geo}-(i) we have $$\textnormal{div}\,\widetilde{\omega}_{j}(t)=0.$$ 
By linearity and uniqueness we have
\begin{equation*}
\omega(t,x)=\sum_{j\ge-1}\widetilde{\omega}_{j}(t,x).
\end{equation*}

\smallskip

Next we shall prove the estimate in (iii). Since $\Delta_j\omega^0=\textnormal{curl}(\Delta_jv^0)$ and $\Delta_jv^0$ is axisymmetric then Proposition \ref{prop:geo} implies that $\Delta_j\omega^0\times e_\theta=(0,0,0)$. Therefore, by Proposition \ref{prop:geo2} we get $\widetilde{\omega}_{j}(t)\times e_\theta=(0,0,0)$ and
$$
\partial_{t}\widetilde{\omega}_{j}+v\cdot\nabla\widetilde{\omega}_{j}= \widetilde{\omega}_{j}\frac{v^r}{r}-\partial_z\Big[\Delta_j\Big(\frac{b^\theta}{r} b\Big)\Big].
$$ 
By the maximum principle, Propositions \ref{prop:gamma} and Bernstein inequality we get
\begin{eqnarray*}
\Vert\widetilde{\omega}_{j}(t)\Vert_{L^\infty}&\le&\Vert\Delta_j\omega^0\Vert_{L^\infty}+\int_0^t \Big\Vert\frac{v^r(\tau)}{r}\Big\Vert_{L^\infty}\Vert\widetilde{\omega}_{j}(\tau)\Vert_{L^\infty}d\tau+\int_0^t \Big\Vert \partial_z\Big[\Delta_j\Big(\frac{b^\theta}{r} b\Big)(\tau)\Big]\Big\Vert_{L^\infty}d\tau\\
&\lesssim& \Vert\Delta_j\omega^0\Vert_{L^\infty}+ \int_0^t \Phi_{1}(\tau)\Vert\widetilde{\omega}_{j}(\tau)\Vert_{L^\infty}d\tau+2^{j}\int_0^t \Big\Vert \Delta_j\Big(\frac{b^\theta}{r}b\Big) (\tau)\Big\Vert_{L^\infty}d\tau.
\end{eqnarray*} 
Then, the Gronwall's inequality implies that
\begin{equation*}
\Vert\widetilde{\omega}_{j}(t)\Vert_{L^\infty}\le C\Big(\Vert\Delta_j\omega^0\Vert_{L^\infty}+2^{j}\Big\Vert \Delta_j\Big(\frac{b^\theta}{r}b\Big)\Big\Vert_{L_{t}^{1}L^\infty}\Big)\Phi_{2}(t).
\end{equation*}
This concludes (iii).

\smallskip

In order to prove (iv), we shall equivalently show that
\begin{equation}\label{t1}
\Vert\Delta_{k}\widetilde{\omega}_{j}(t)\Vert_{L^\infty}\lesssim 2^{k-j}e^{C V(t)}\Big(\Vert \Delta_{j}\omega^{0}\Vert_{L^\infty} + 2^{j}\Big\Vert \Delta_j\Big(\frac{b^\theta}{r}b\Big)\Big\Vert_{L_{t}^{1}L^\infty}\Big)
\end{equation}
and
\begin{equation}\label{t2}
\Vert\Delta_{k}\widetilde{\omega}_{j}(t)\Vert_{L^\infty}\lesssim 2^{j-k}e^{C V(t)}\Big(\Vert\Delta_{j}\omega^{0}\Vert_{L^\infty} + 2^{j}\Big\Vert\Delta_j \Big(\frac{b^\theta}{r}b\Big)\Big\Vert_{L_{t}^{1}L^\infty}\Big).
\end{equation}

\smallskip

{\bf Proof of \eqref{t1}}. Applying Proposition \ref{prop:trans-besov} to the equation \eqref{rt} yields
\begin{eqnarray}\label{y}
\nonumber e^{-CV(t)}\Vert\widetilde{\omega}_{j}(t)\Vert_{B^{-1}_{\infty,\infty}}&\lesssim& \Vert\Delta_{j}\omega^0\Vert_{B^{-1}_{\infty,\infty}}+ \int_0^t e^{-CV(\tau)}\Vert\widetilde{\omega}_{j}.\nabla v(\tau)\Vert_{B^{-1}_{\infty,\infty}}d\tau\\
&+& \int_0^t e^{-CV(\tau)}\Big\Vert \partial_z\Big[\Delta_j\Big(\frac{b^\theta}{r} b\Big)(\tau)\Big] \Big\Vert_{B^{-1}_{\infty,\infty}}d\tau,
\end{eqnarray}
where $V(t)=\int_{0}^{t}\Vert v(\tau)\Vert_{B^{1}_{\infty,1}}d\tau.$
Using Bony's decomposition we can easily check that
\begin{equation}\label{z}
\Vert\widetilde{\omega}_{j}.\nabla v\Vert_{B^{-1}_{\infty,\infty}}\lesssim \Vert v \Vert_{B_{\infty,1}^1} \Vert\widetilde{\omega}_{j}\Vert_{B^{-1}_{\infty,\infty}}.
\end{equation}
On the other hand, by using  H\"older and Bernstein inequalities  we obtain
\begin{eqnarray*}
\int_0^t e^{-CV(\tau)}\Big\Vert \partial_z\Big[\Delta_j\Big(\frac{b^\theta}{r} b\Big)\Big](\tau)\Big\Vert_{B^{-1}_{\infty,\infty}}d\tau 
&\le& \Big\Vert \partial_z\Big[\Delta_j\Big(\frac{b^\theta}{r} b\Big)\Big]\Big\Vert_{L_{t}^{1}B^{-1}_{\infty,\infty}} \\
&\lesssim& \Big\Vert \Delta_j\Big(\frac{b^\theta}{r}b\Big)\Big\Vert_{L_{t}^{1} B^{0}_{\infty,\infty}}\\
&\lesssim& \Big\Vert \Delta_j\Big(\frac{b^\theta}{r}b\Big) \Big\Vert_{L_{t}^{1} L^\infty}.
\end{eqnarray*}
Inserting the two last estimates into \eqref{y} yields,
\begin{equation*}
e^{-CV(t)}\Vert\widetilde{\omega}_{j}(t)\Vert_{B^{-1}_{\infty,\infty}}\lesssim \Vert\Delta_{j}\omega^0\Vert_{B^{-1}_{\infty,\infty}}+ \Big\Vert \Delta_j\Big(\frac{b^\theta}{r}b\Big) \Big\Vert_{L_{t}^{1} L^\infty}+\int_0^t e^{-CV(\tau)}\Vert\widetilde{\omega}_{j}\Vert_{B^{-1}_{\infty,\infty}}\Vert v(\tau)\Vert_{B_{\infty,1}^1} d\tau.
\end{equation*}
Using Gronwall's inequality gives
\begin{eqnarray*}
\Vert\widetilde{\omega}_{j}(t)\Vert_{B^{-1}_{\infty,\infty}}&\lesssim& \Big(\Vert\Delta_{j}\omega^0\Vert_{B^{-1}_{\infty,\infty}}+ \Big\Vert \Delta_j\Big(\frac{b^\theta}{r}b\Big) \Big\Vert_{L_{t}^{1} L^\infty}\Big) e^{CV(t)}\\
&\lesssim& \Big(2^{-j} \Vert\Delta_{j}\omega^0\Vert_{L^{\infty}}+ \Big\Vert \Delta_j\Big(\frac{b^\theta}{r}b\Big) \Big\Vert_{L_{t}^{1} L^\infty}\Big) e^{CV(t)}. 
\end{eqnarray*}
Finally, by definition, we deduce  that
\begin{equation*}
\Vert\Delta_{k}\widetilde{\omega}_{j}(t)\Vert_{L^{\infty}}\lesssim 2^{k-j} e^{CV(t)}\Big(\Vert\Delta_{j}\omega^0\Vert_{L^{\infty}} +2^{j}  \Big\Vert \Delta_j\Big(\frac{b^\theta}{r}b\Big) \Big\Vert_{L_{t}^{1} L^\infty}\Big).
\end{equation*}
which is the desired result.

\smallskip

{\bf Proof of \eqref{t2}}. In this case it is   more convenient to work with the cartesian coordinates of   $\widetilde{\omega}_{j}=(\widetilde{\omega}_{j}^1,\widetilde{\omega}_{j}^2,0).$ The analysis will be exactly the same for both
components, so we will deal only with the first one.
Since $v^\theta=0$ the one may easily check that
 $$\frac{v^r}{r}=\frac{v^1}{x_1}=\frac{v^2}{x_2}. 
 $$ 
 Then,  the functions $\widetilde{\omega}_{j}^1$ solves the equation
\begin{equation*} 
\left\{\begin{array}{ll} 
\partial_{t}\widetilde{\omega}_{j}^1+v\cdot\nabla\widetilde{\omega}_{j}^1= v^2\frac{\widetilde{\omega}_{j}^1}{x_2}-\partial_z\Big[\Delta_j\Big(\frac{b^\theta}{r} b_1\Big)\Big], \\
\widetilde{\omega}_{j}^{1}(t=0)=\Delta_{j}\omega_{0}^{1}.  
\end{array} \right.
\end{equation*} 
Applying Proposition \ref{prop:trans-besov} to this last equation we get
\begin{eqnarray}\label{z1}
\nonumber e^{-CV(t)}\Vert\widetilde{\omega}_{j}^1(t)\Vert_{B^{1}_{\infty,1}}&\lesssim& \Vert\Delta_{j}\omega^1_0\Vert_{B^{1}_{\infty,1}}+ \int_0^t e^{-CV(\tau)}\Big\Vert v^2\frac{\widetilde{\omega}_{j}^1(\tau)}{x_2}\Big\Vert_{B^{1}_{\infty,1}}d\tau\\
&+& \int_0^t e^{-CV(\tau)}\Big\Vert \partial_z\Big[\Delta_j\Big(\frac{b^\theta}{r} b_1\Big)\Big](\tau)\Big\Vert_{B^{1}_{\infty,1}}d\tau.
\end{eqnarray}
We shall make use of the following estimate, proved in pages 31-32 from \cite{AHK}:
\begin{equation}\label{z13}
\Big\Vert v^2\frac{\widetilde{\omega}_{j}^1}{x_2}\Big\Vert_{B^{1}_{\infty,1}}\lesssim \Vert v \Vert_{B^{1}_{\infty,1}}\Big(\Vert\widetilde{\omega}^{1}_{j}\Vert_{B^{1}_{\infty,1}}+\Big\Vert\frac{\widetilde{\omega}^{1}_{j}}{x_2}\Big\Vert_{B^{0}_{\infty,1}}\Big). 
\end{equation}
Now we will estimate $\Big\Vert\frac{\widetilde{\omega}^{1}_{j}}{x_2}\Big\Vert_{L_{t}^{\infty} B^{0}_{\infty,1}}.$ By Proposition \ref{prop:geo} we have $\widetilde{\omega}^{1}_{j}(x_1,0,z)=0.$ Then,  by Taylor formula we get  
$$\frac{\widetilde{\omega}^{1}_{j}(x_1,x_2,z)}{x_2}=\int_{0}^{1}\partial_{x_{2}} \widetilde{\omega}^{1}_{j}(x_1,\tau x_2,z) d\tau.$$
At this stage we need to the following estimate, describing  the anisotropic dilatation in Besov space $B_{\infty,1}^0$, see Proposition A.1 from \cite{AHK}:
 for every $0<\tau<1$, we have
\be\label{prop-a5}
\Vert h_{\tau}\Vert_{B_{\infty,1}^0}\le C(1-\log \tau)\Vert h \Vert_{B_{\infty,1}^0},
\ee
where $h_{\tau}(x_1,x_2,x_3)=h(x_1,\tau x_2,x_3)$ and  $C$ is a absolute positive constant.
It follows that
\begin{align*}
\nonumber\Big\Vert\frac{\widetilde{\omega}^{1}_{j}}{x_2}\Big\Vert_{B^{0}_{\infty,1}}\le \Vert\partial_{x_{2}} \widetilde{\omega}^{1}_{j}\Vert_{B^{0}_{\infty,1}}\int_{0}^{1}(1-\log \tau)d\tau 
\lesssim \Vert\widetilde{\omega}^{1}_{j}\Vert_{B^{1}_{\infty,1}}.
\end{align*}
Inserting this estimate into \eqref{z13} gives
\begin{equation*}
\Big\Vert v^2\frac{\widetilde{\omega}_{j}^1}{x_2}\Big\Vert_{B^{1}_{\infty,1}}\lesssim \Vert v \Vert_{B^{1}_{\infty,1}}\Vert\widetilde{\omega}^{1}_{j}\Vert_{B^{1}_{\infty,1}}. 
\end{equation*}
On the other hand, using Bernstein inequalities we find 
\begin{eqnarray*}
\int_0^t e^{-CV(\tau)}\Big\Vert \partial_z\Big[\Delta_j\Big(\frac{b^\theta}{r} b_1\Big)\Big](\tau) \Big\Vert_{B^{1}_{\infty,1}}d\tau &\le& \Big\Vert  \partial_z\Big[\Delta_j\Big(\frac{b^\theta}{r} b_1\Big)\Big]\Big\Vert_{L_{t}^{1}B^{1}_{\infty,1}} \\
&\lesssim& \Big\Vert \Delta_j\Big(\frac{b^\theta}{r} b_1\Big)\Big\Vert_{L_{t}^{1} B^{2}_{\infty,1}}\\
&\lesssim& 2^{2j} \Big\Vert \Delta_j\Big(\frac{b^\theta}{r} b_1\Big)\Big\Vert_{L_{t}^{1} L^\infty}.
\end{eqnarray*}
Plugging the two last estimates into \eqref{z1}  we conclude that
\begin{align*}
 e^{-CV(t)}\Vert\widetilde{\omega}_{j}^1(t)\Vert_{B^{1}_{\infty,1}}&\lesssim \Vert\Delta_{j}\omega^1_0\Vert_{B^{1}_{\infty,1}}+\int_0^t e^{-CV(\tau)}\Vert\widetilde{\omega}_{j}^1(\tau)\Vert_{B^{1}_{\infty,1}}\Vert v(\tau)\Vert_{B^{1}_{\infty,1}}d\tau +
 2^{2j} \Big\Vert \Delta_j\Big(\frac{b^\theta}{r} b_1\Big)\Big\Vert_{L_{t}^{1} L^\infty}.
\end{align*}
Then Gronwall's inequality implies that
\begin{equation*}
\Vert\widetilde{\omega}^{1}_{j}(t)\Vert_{B^{1}_{\infty,1}}\lesssim \Big(2^{j}\Vert\Delta_{j}\omega_0\Vert_{L^\infty}+2^{2j}\Big\Vert \Delta_{j}\Big(\frac{b^\theta}{r}b_1\Big)\Big\Vert_{L_{t}^{1} L^{\infty}}\Big)e^{CV(t)}.
\end{equation*}
This can be written as 
\begin{equation*}
\Vert\Delta_{k}\widetilde{\omega}^{1}_{j}(t)\Vert_{L^{\infty}}\lesssim 2^{j-k} e^{CV(t)}\Big(\Vert\Delta_{j}\omega_0\Vert_{L^\infty}+2^{j}\Big\Vert \Delta_{j}\Big(\frac{b^\theta}{r}b_1\Big)\Big\Vert_{L_{t}^{1} L^{\infty}}\Big),
\end{equation*}
ending the proof of the proposition.
\end{proof}

\subsection{Lipschitz bound}
The following result deals with the global regularity persistence  for initial data belonging to the critical Besov space $B_{\infty,1}^1$.
\begin{proposition}\label{prop-a6} Let $v_0\in L^2\cap B^{1}_{\infty,1}$ such that ${\omega^\theta_0}/{r}\in L^{3,1}$ and $b_0\in L^2\cap B^{\frac{3}{\sigma}}_{\sigma,1}$,  $\sigma\in (3,\infty)$,  such that ${b^\theta_0}/{r}\in L^2\cap L^\infty$. Then 
for every $t \in \R_+,$ 
$$
\Vert\omega(t)\Vert_{B^{0}_{\infty,1}}+\Vert v(t)\Vert_{B^{1}_{\infty,1}}\leq \Phi_{4}(t).
$$
\end{proposition}
\begin{proof}
Let $N$ be a fixed positive integer that will be chosen later. By definition of Besov spaces and in view of Proposition \ref{prop-a4}-{(i)} one has
\begin{eqnarray}\label{s11}
\nonumber \Vert\omega(t)\Vert_{B^{0}_{\infty,1}}&\le& \sum_{k}\Vert\Delta_{k}\sum_{j}\widetilde{\omega}_{j}(t)\Vert_{L^{\infty}}\\
\nonumber &\lesssim& \sum_{k}\bigg(\sum_{\vert k-j \vert \ge N}\Vert\Delta_{k}\widetilde{\omega}_{j}(t)\Vert_{L^{\infty}}\bigg)+\sum_{k}\bigg(\sum_{\vert k-j \vert < N}\Vert\Delta_{k}\widetilde{\omega}_{j}(t)\Vert_{L^{\infty}}\bigg)\\ 
&:=& \textnormal{I}_{1}+\textnormal{I}_{2}.
\end{eqnarray}
To estimate the term $\textnormal{I}_{1}$ we use Proposition \ref{prop-a4}-{(iv)},\begin{eqnarray}\label{s2}
\nonumber \textnormal{I}_{1}&=& \sum_{k}\bigg(\sum_{\vert k-j \vert \ge N}\Vert\Delta_{k}\widetilde{\omega}_{j}(t)\Vert_{L^{\infty}}\bigg)\\ 
 &\lesssim& 2^{-N}\bigg(\Vert\omega_{0}\Vert_{B^{0}_{\infty,1}}+\Big\Vert\frac{b^\theta}{r}b\Big\Vert_{L_{t}^{1}B^{1}_{\infty,1}}\bigg)e^{C V(t)}.
\end{eqnarray}
For the second term $\textnormal{I}_{2}$ we use the continuity of the operator $\Delta_k$ in the space $L^\infty$ and Proposition \ref{prop-a4}-{(iii)}, \begin{eqnarray}\label{s3}
\nonumber \textnormal{I}_{2}&\lesssim& \sum_{j}\sum_{\vert k-j \vert < N}\Vert\widetilde{\omega}_{j}(t)\Vert_{L^{\infty}}\\
 &\lesssim& N\Phi_{2}(t)\bigg(\Vert\omega_{0}\Vert_{B^{0}_{\infty,1}}+\Big\Vert\frac{b^\theta}{r} b\Big\Vert_{L_{t}^{1}B^{1}_{\infty,1}}\bigg).
\end{eqnarray}
This, with Proposition \ref{prop:energy}-(ii),  Proposition \ref{prop:app}-(ii), the embedding $B_{{\sigma},1}^{\frac{3}{\sigma}+1}\hookrightarrow B_{\infty,1}^1$ and Proposition \ref{propo:bound-vort}, yields
\begin{equation}\label{prodbbbes}
\Big\Vert \frac{b^\theta}{r}b\Big\Vert_{L_t^1 B^{1}_{\infty,1}}\lesssim \Big\Vert \frac{b^\theta}{r}\Big\Vert_{L_t^\infty L^{\infty}}\Vert b\Vert_{L_t^1 B^{1}_{\infty,1}}\leq \Phi_2(t). 
\end{equation}
Now combining  \eqref{s11}, \eqref{s2} and  \eqref{s3} gives
\begin{eqnarray*}
\Vert\omega(t)\Vert_{B^{0}_{\infty,1}}\lesssim \big(2^{-N}e^{C V(t)}+\Phi_{2}(t)N\big)\Big(\Vert\omega_0(t)\Vert_{B^{0}_{\infty,1}}+\Phi_{2}(t)\Big).
\end{eqnarray*}
We choose $N$ such that
$$N=[C V(t)]+1,$$ we obtain
\begin{equation}\label{s4}
\Vert\omega(t)\Vert_{B^{0}_{\infty,1}}\leq (V(t)+1) \Phi_{2}(t).
\end{equation}
It remains to estimate $V(t)$. For this purpose we have by definition of Besov space, Bernstein inequality and the estimate $\Vert\Delta_{j}v \Vert_{L^\infty}\sim 2^{-j}\Vert\Delta_{j}\omega \Vert_{L^\infty},$
\begin{eqnarray*}
\Vert v(t)\Vert_{B^{1}_{\infty,1}}&=& \sum_{j \ge -1} 2^{j}\Vert\Delta_{j}v(t)\Vert_{L^\infty}\\
&\lesssim& \Vert\Delta_{-1}v(t)\Vert_{L^\infty}+\sum_{j \ge 0}\Vert\Delta_{j}\omega(t)\Vert_{L^\infty}\\
&\lesssim& \Vert v(t)\Vert_{L^\infty}+\Vert\omega(t)\Vert_{B^{0}_{\infty,1}}\\
&\leq& \Phi_{3}(t)\Big(1+\int_{0}^{t}\Vert v(\tau)\Vert_{B^{1}_{\infty,1}}d\tau\Big). 
\end{eqnarray*}
where we have used in the last inequality the estimate  \eqref{s4} and Proposition \ref{propo:bound-vel}.  Then, an application of the Gronwall inequality gives rise to
\begin{equation}\label{s5}
\Vert v(t)\Vert_{B^{1}_{\infty,1}}\lesssim \Phi_{4}(t).
\end{equation}
Plugging this estimate into \eqref{s4} gives,
$$\Vert\omega(t)\Vert_{B^{0}_{\infty,1}}\leq \Phi_{4}(t).$$
Using now the embeddings $B^{1}_{\infty,1}\hookrightarrow \textnormal{Lip}(\R^3)$ and \eqref{s5} we get
\be\label{lipes}
\Vert\nabla v(t)\Vert_{L^\infty}\lesssim \Phi_{4}(t).
\ee 
This concludes  the proof of the proposition.\qed 
\end{proof}

\vspace{0.2cm}

Now we will propagate the critical Besov $B_{p,1}^{\frac3p+1}$ regularities globally in time for $2\leq p< \infty$.
 More precisely, we prove the following proposition.
 
 \begin{proposition}\label{prop-a7} Let  $(v_0,b_0)\in ( L^2\cap B_{p,1}^{\frac3p+1})\times 
( L^2\cap B_{p,1}^{\frac3p-1})$ with  $2\le p< \infty$. 
Assume, in addition that   
${\omega_0^\theta}/{r}\in L^{3,1}$ and ${b_0^\theta}/{r}\in L^2 \cap L^{\infty}$.  
Then 
for every $t \in \R_+,$ 
$$
\Vert v\Vert_{{L}^\infty_t  B_{p,1}^{1+\frac3p}}+\Vert b\Vert_{{L}^\infty_t B_{p,1}^{\frac3p-1}}+\Vert b\Vert_{{L}^1_t B_{{{p}},1}^{\frac{3}{{{p}}}+1}}\leq \Phi_{6}(t).$$
\end{proposition}
\begin{proof}
We have by definition
\begin{align}\label{vbp0} 
\nonumber\Vert v(t)\Vert_{B^{1+\frac3p}_{p,1}}&= \sum_{j \ge -1} 2^{j(1+\frac3p)}\Vert\Delta_{j}v(t)\Vert_{L^p}\\
\nonumber&\lesssim \Vert\Delta_{-1}v(t)\Vert_{L^p}+\sum_{j \ge 0}2^\frac3p\Vert\Delta_{j}\omega(t)\Vert_{L^p}\\
&\lesssim \Vert v(t)\Vert_{L^p}+\Vert\omega(t)\Vert_{B^{\frac3p}_{p,1}}.
\end{align}
We shall first estimate   $\Vert\omega(t)\Vert_{B^{\frac3p}_{p,1}}$.  For this purpose we apply Proposition \ref{prop:trans-besov} to the vorticity equation \eqref{eq:vorticity}, we obtain
\begin{equation*}
 e^{-C U(t)}\Vert \omega(t) \Vert_{B_{p,1}^\frac3p}\lesssim \Vert \omega_0 \Vert_{B_{p,1}^\frac3p}+\int_0^t e^{-C U(\tau)}\Vert\omega\cdot\nabla v(\tau) \Vert_{B_{p,1}^\frac3p}d\tau+\int_0^t e^{-C U(\tau)}\Vert \textnormal{curl}\big(b\cdot\nabla b)(\tau) \Vert_{B_{p,1}^\frac3p} d\tau,
\end{equation*}
for all $p\in [1,\infty)$ where $U(t):=\Vert\nabla v\Vert_{L_t^1L^\infty}$. Since $\omega=\textnormal{curl}\,v$ then using Bony's decomposition we can show that (see page 36 from \cite{AHK})
$$
\Vert \omega\cdot\nabla v(\tau) \Vert_{B_{p,1}^\frac3p}\lesssim \Vert \omega\Vert_{B_{p,1}^\frac3p}\Vert\nabla  v(\tau) \Vert_{L^\infty}.
$$
Combining the two last estimates and using Gronwall's inequality we find
\begin{equation*}
\Vert \omega(t) \Vert_{B_{p,1}^\frac3p}\lesssim \Big(\Vert \omega_0 \Vert_{B_{p,1}^\frac3p}+\int_0^t e^{-C U(\tau)}\Vert \textnormal{curl}\big(b\cdot\nabla b)(\tau) \Vert_{B_{p,1}^\frac3p} d\tau \Big)e^{C U(t)}.
\end{equation*}
Since $\textnormal{curl} (b\cdot\nabla b)=-\partial_z \Big(\frac{ b^\theta}{r}b\Big)$ then using Proposition \eqref{prop:app}-(ii) and Proposition \eqref{prop:energy}-(ii), we readily  obtain
$$
\Vert \textnormal{curl}\big(b\cdot\nabla b)\Vert_{B_{p,1}^\frac3p}\lesssim \Big\Vert \frac{b^\theta}{r}   b \Big\Vert_{B_{p,1}^{\frac3p+1}}\lesssim \Big\Vert \frac{b^\theta}{r}   \Big\Vert_{L^\infty}\Vert b \Vert_{B_{p,1}^{\frac3p+1}}\lesssim\Big\Vert \frac{b^\theta_0}{r}   \Big\Vert_{L^\infty}\Vert b \Vert_{B_{p,1}^{\frac3p+1}}.
$$ 
Combining the two last estimate, we find , 
\begin{equation}\label{omegabp}
\Vert \omega(t) \Vert_{B_{p,1}^\frac3p}\lesssim \Big(\Vert \omega_0 \Vert_{B_{p,1}^\frac3p}+\Big\Vert \frac{b^\theta_0}{r}   \Big\Vert_{ L^\infty} \Vert b \Vert_{L^1_t B_{p,1}^{\frac3p+1}}\Big)
e^{C U(t)}.
\end{equation}
Now we will estimate $\Vert v(t)\Vert_{L^p}$.
 Since the Riesz transforms act continuously on $L^p$, for all $1<p<\infty$, then  from \eqref{veleq} 
we get
\begin{align*}
\Vert v(t)\Vert_{L^p}&\lesssim \Vert v_{0}\Vert_{L^{p}}+\int_0^t\Vert b\cdot \nabla b (\tau)\Vert_{L^{p}}d\tau +\int_0^t\Vert v (\tau)\Vert_{L^{p}}\Vert \nabla v (\tau)\Vert_{L^{\infty}}d\tau.
\end{align*}
Using   Gronwall's inequality, the identity  \eqref{bnablab}, the estimate \eqref{normgammalp} and  the embedding $B_{p,1}^{\frac3p+1}\hookrightarrow L^p $ we obtain, for all $p\in (1,\infty)$, 
\begin{align}\label{vlp2}
\Vert v(t)\Vert_{L^p}
&\le \Big(\Vert v_{0}\Vert_{L^{p}}+\int_0^t\Vert b\cdot \nabla b (\tau)\Vert_{L^{p}}d\tau \Big)e^{C\Vert\nabla v\Vert_{L^1_tL^{\infty}}}
\nonumber\\
&\lesssim \Big(\Vert v_{0}\Vert_{L^{p}}+\Vert {b^\theta }/{r}\Vert_{L^\infty_t L^\infty}\Vert b \Vert_{L^1_tL^{p}} \Big)e^{C\Vert\nabla v\Vert_{L^1_tL^{\infty}}}
\nonumber\\
&\lesssim \Big(\Vert v_{0}\Vert_{L^{p}}+\Vert {b_0^\theta }/{r}\Vert_{ L^\infty}\Vert b \Vert_{L^1_t B_{p,1}^{\frac3p+1}} \Big)e^{C\Vert\nabla v\Vert_{L^1_tL^{\infty}}}.
\end{align}
Inserting  \eqref{omegabp}  and \eqref{vlp2} into \eqref{vbp0} gives
 \begin{align}\label{vbp10}
\Vert v(t)\Vert_{B^{1+\frac3p}_{p,1}}
&\lesssim \Big(\Vert v_{0}\Vert_{L^{p}}+\Vert \omega_0 \Vert_{B_{p,1}^\frac3p}+\Big\Vert \frac{b^\theta_0}{r}   \Big\Vert_{ L^\infty} \Vert b \Vert_{L^1_t B_{p,1}^{\frac3p+1}} \Big)e^{C U(t)}.
\end{align}
Now, to estimate the term $\Vert b \Vert_{L^1_t B_{p,1}^{\frac3p+1}}$, we distinguish two cases: 
\paragraph{Case $p>3$.} 
Using  the embedding $ B_{\infty,1}^{1}\hookrightarrow \textnormal{Lip}(\R^3)$,  Proposition \ref{propo:bound-vort},   and  Proposition \ref{prop-a6},  we readily  conclude  that
\begin{align*}
\Vert b\Vert_{L_t^\infty L^{p}} + \Vert b\Vert_{\widetilde{L}^1_t B_{{p},1}^{\frac{3}{p}+1}} +\Vert v(t)\Vert_{B^{1+\frac3p}_{p,1}}\le \Phi_5(t). 
\end{align*}
 This achieves the proof for  the case $p>3$.

 \paragraph{Case $2\le p\le 3$.}
 In view of \eqref{decoup-b0} we have
\begin{align}\label{est:bfin-sec}
\nonumber\Vert b\Vert_{\widetilde{L}^\infty_t B_{p,1}^{\frac3p-1}}+\Vert b\Vert_{\widetilde{L}^1_t B_{p,1}^{\frac3p+1}} &\lesssim \Vert\Delta_{-1} b\Vert_{L_t^1 L^p}+\Vert b_0\Vert_{B_{p,1}^{\frac3p-1}}+\Vert b\cdot\nabla v\Vert_{\widetilde{L}^1_t B_{p,1}^{\frac3p-1}}\\ &\quad +\sum_{q\geq 0}2^{q(\frac{3}{p}-1)}\big\Vert[\Delta_q,v\cdot\nabla]b\big\Vert_{L^1_t L^{p}}.
\end{align}
From Lemma \ref{m1}, since $-1<\frac{3}{p}-1<\frac12$ then one has
\begin{align}\label{comfin00}
\sum_{q\geq -1}2^{q(\frac{3}{p}-1)}\big\Vert[\Delta_q,v\cdot\nabla]b\big\Vert_{L^1_t L^{p}}
&\lesssim \int_0^t\Vert \nabla v (\tau)\Vert_{ L^\infty}\Vert b(\tau)\Vert_{ B_{p,1}^{\frac{3}{p}-1}}d\tau.
\end{align}
 In order to estimate the term $\Vert b\cdot\nabla v\Vert_{\widetilde{L}^1_t B_{p,1}^{\frac3p-1}}$, in \eqref{est:bfin-sec}, we shall use the following product law
\begin{align*}
\Vert b\cdot\nabla v\Vert_{ B_{p,1}^{\frac3p-1}} &\lesssim 
\left\{ 
\begin{array}{ll} 
\Vert\nabla v\big\Vert_{L^\infty}\Vert b\Vert_{ B_{p,1}^{\frac3p-1}}+\Vert b\Vert_{L^\infty}\Vert   v\Vert_{ B_{p,1}^{\frac3p}} &\quad\hbox{if}\quad p<  3,\\
\Vert   v\Vert_{ B_{\infty,1}^{1}}\Vert b\Vert_{ B_{p,1}^{\frac3p-1}} &\quad\hbox{if}\quad p=3,
\end{array} \right.
\end{align*}
where the first case follows from the fact that $B_{p,1}^{\frac3p-1}\cap L^\infty$ is an algebra for every $p<3$ and the second one   is proved in \cite{HKB}.
Thus, using  the embeddings $B_{\infty,1}^{1} \hookrightarrow \textnormal{Lip}(\R^3)$ and $B_{p,1}^{\frac3p+1} \hookrightarrow   B_{p,1}^{\frac3p}$ we conclude that
\be\label{prodfin-sec}
\Vert b\cdot\nabla v\Vert_{ B_{p,1}^{\frac3p-1}} \lesssim \Vert   v\Vert_{ B_{\infty,1}^{1}}\Vert b\Vert_{ B_{p,1}^{\frac3p-1}}+\Vert b\Vert_{L^\infty}\Vert   v\Vert_{ B_{p,1}^{\frac3p+1}} .
\ee
Inserting \eqref{comfin00} and   \eqref{prodfin-sec} into \eqref{est:bfin-sec} and using the embedding $B_{p,1}^{\frac3p-1}\hookrightarrow L^p$ 
we obtain
\begin{align*}
\Vert b\Vert_{\widetilde{L}^\infty_t B_{p,1}^{\frac3p-1}}+\Vert b\Vert_{\widetilde{L}^1_t B_{p,1}^{\frac3p+1}} &\lesssim \Vert b_0\Vert_{B_{p,1}^{\frac3p-1}}+\int_0^t\Vert b(\tau)\Vert_{ L^\infty}\Vert   v(\tau)\Vert_{B_{p,1}^{\frac3p+1}}d\tau+\int_0^t\big(1+\Vert   v(\tau)\Vert_{  B_{\infty,1}^{1}}\big)\Vert b(\tau)\Vert_{ B_{p,1}^{\frac3p-1}}d\tau.
\end{align*}
Then, Gronwall's inequality 
gives
\begin{align}\label{bbp11}
\Vert b\Vert_{\widetilde{L}^\infty_t B_{p,1}^{\frac3p-1}}+\Vert b\Vert_{\widetilde{L}^1_t B_{p,1}^{\frac3p+1}}&\lesssim \Big( \Vert b_0\Vert_{B_{p,1}^{\frac3p-1}}+\int_0^t\Vert b(\tau)\Vert_{ L^\infty}\Vert   v(\tau)\Vert_{B_{p,1}^{\frac3p+1}}d\tau\Big)e^{C\Vert v\Vert_{L^t_1 B_{\infty,1}^{1}}}.
\end{align}
Plugging \eqref{bbp11} into \eqref{vbp10} we deduce that
\begin{align*}
\notag\Vert v(t)\Vert_{B^{1+\frac3p}_{p,1}}
\lesssim 
\Big(&\Vert v_{0}\Vert_{L^{p}}+\Vert \omega_0 \Vert_{B_{p,1}^\frac3p}+\Big\Vert \frac{b^\theta_0}{r}   \Big\Vert_{ L^\infty} \Vert b_0\Vert_{B_{p,1}^{\frac3p-1}}+\Big\Vert \frac{b^\theta_0}{r}   \Big\Vert_{ L^\infty} \int_0^t\Vert b(\tau)\Vert_{ L^\infty}\Vert   v(\tau)\Vert_{B_{p,1}^{\frac3p+1}}d\tau\Big)e^{C\Vert v\Vert_{L^t_1 B_{\infty,1}^{1}}}.
\end{align*}
Applying again  Granwall's inequality we get
\begin{align*}
\notag\Vert v(t)\Vert_{B^{1+\frac3p}_{p,1}}
\leq C_0e^{C\Vert v\Vert_{L^t_1 B_{\infty,1}^{1}}}\textnormal{exp}\Big(C_0 \Vert b\Vert_{L^t_1 L^\infty}  e^{C\Vert v\Vert_{L^t_1 B_{\infty,1}^{1}}}\Vert b\Vert_{L^t_1 L^\infty}\Big).
\end{align*}
Finally, using the embedding $B^{\frac3p-1}_{p,1} \hookrightarrow B^{\frac{3}{\sigma}-1}_{\sigma,1}$, for some $\sigma>3$,  Proposition \ref{prop:energy}-(iii) and  Proposition \ref{prop-a6} we conclude that
\begin{align*}
\notag\Vert v(t)\Vert_{B^{1+\frac3p}_{p,1}}
\le \Phi_6(t). 
\end{align*}
which give in turn, by \eqref{bbp11},
\begin{align*}
\Vert b\Vert_{\widetilde{L}^\infty_t B_{p,1}^{\frac3p-1}}+\Vert b\Vert_{\widetilde{L}^1_t B_{p,1}^{\frac3p+1}}&\le\Phi_6(t).
\end{align*}
 This achieves the proof of the Proposition.\qed

\end{proof}
 
\appendix

\section{Appendix}
Our task now is to prove the following product laws.
\begin{proposition}\label{prop:app}
Let $b$ be a smooth axisymmetric vector field with a trivial radial component;  $b^r=0$. Then the following estimate occurs
\begin{enumerate}[label=\rm(\roman*)]
\item For all $s>0$ and $t>0$ we have
$$\Big\Vert \frac{ b^\theta}{r}b\Big\Vert_{\widetilde{L}^1 H^{s}}\leq C\Big\Vert \frac{ b^\theta}{r}\Big\Vert_{L_t^\infty L^\infty}\Vert b\Vert_{\widetilde{L}^1 H^{s}}.$$
\item For all $p\in [1,\infty]$ we have
 \begin{equation*}
\Big\Vert \frac{b^\theta}{r}b\Big\Vert_{B^{\frac3p+1}_{p,1}}\lesssim \Big\Vert \frac{b^\theta}{r} \Big\Vert_{L^\infty}\Vert b \Vert_{B^{\frac3p+1}_{p,1}}. 
\end{equation*}
\end{enumerate}
\end{proposition}
\begin{proof}
Using  Bony’s decomposition \eqref{j}, we write 
\begin{eqnarray}\label{hz2}
\frac{b^\theta}{r} b&=&T_{\frac{b^\theta}{r}}b+ T_{b}\frac{b^\theta}{r}+  \mathcal{R}\Big(\frac{b^\theta}{r},b\Big).
\end{eqnarray}
By definition of the paraproduct, we have 
\be\label{gt}
T_{\frac{b^\theta}{r}}b=\sum_{k} S_{k-1}\Big(\frac{b^\theta}{r}\Big)\Delta_{k} b.
\ee
Then, by  \eqref{orth} for all $k\in\mathbb{N}\cup\{-1\}$ one has
$$\Delta_{q}\big(T_{\frac{b^\theta}{r}}b\big)=\sum_{\vert k-q \vert \le 4}\Delta_{q}\Big(S_{k-1}\frac{b^\theta}{r}\Delta_{k}b\Big).
$$
Using H\"older  inequality, we get
\begin{equation}\label{tgood}
\big\Vert \Delta_{q}\big(T_{\frac{b^\theta}{r}}b\big)\big\Vert_{{L}^1_t L^2}
\le  \Big\Vert \frac{b^\theta}{r}\Big\Vert_{{L}^\infty_t L^\infty}\sum_{\vert k-q \vert \le 4}\Vert\Delta_{q} b\Vert_{{L}^1_t L^2}.
\end{equation}
It follows that
\begin{eqnarray}\label{hs1}
\big\Vert T_{\frac{b^\theta}{r}}b\big\Vert_{\widetilde{L}^1_t H^{s}}=\Big\| 2^{qs}\big\Vert \Delta_{q}\big(T_{\frac{b^\theta}{r}}b\big)\big\Vert_{{L}^1_t L^2}\Big\|_{\ell^2}
&\lesssim& \Big\Vert \frac{b^\theta}{r}\Big\Vert_{{L}^\infty_tL^\infty} \Vert b\Vert_{\widetilde{L}^1_t H^{s}}.
\end{eqnarray}
For the remainder term we have by definition and by \eqref{orth},
\begin{eqnarray}\label{gr}
\Delta_{j}\mathcal{R}\Big(b,\frac{b^\theta}{r}\Big)
&=& \sum_{k\geq j-4}\Delta_{j}\Big(\Delta_{k}b\,\widetilde{\Delta}_{k}\Big(\frac{b^\theta}{r}\Big)\Big).
\end{eqnarray}
  It follows that
\begin{align*}
\Big\Vert\Delta_{j}\mathcal{R}\big(b,\frac{b^\theta}{r}\big)\Big\Vert_{{L}^1_tL^2}&\lesssim 
\sum_{k\ge j-4} \Vert\Delta_{k}b\Vert_{{L}^1_tL^2} \Big\Vert\widetilde{\Delta}_{k}\Big(\frac{b^\theta}{r}\Big)\Big\Vert_{{L}^\infty_tL^{\infty}},
\end{align*}
where we have used the continuity of the operator $\Delta_{j}$ in $L^2.$ Therefore
\begin{eqnarray}\label{hz4}
\nonumber \Big\Vert\Big(2^{js}\Big\Vert\Delta_{j}\mathcal{R}\big(b,\frac{b^\theta}{r}\big)\Big\Vert_{L^2}\Big)_k\Big\|_{\ell^2}
\nonumber&\lesssim&  \Big\Vert\frac{b^\theta}{r}\Big\Vert_{{L}^\infty_t L^{\infty}}\Big(\sum_{j\geq -1}\Big(\sum_{k\ge j-4}2^{(j-k)s}2^{ks}\Vert\Delta_{k}b\Vert_{{L}^1_tL^2}\Big)^2\Big)^\frac12\\
 &\lesssim& \Big\Vert\frac{b^\theta}{r}\Big\Vert_{{L}^\infty_t L^{\infty}}\Vert b \Vert_{\widetilde{L}^1_t H^s}.
\end{eqnarray}
To estimate the term $\Big\Vert T_{b}\frac{b^\theta}{r}\Big\Vert_{H^s},$ we use the axisymmetric structure of the vector field $b$. We write 
$$\Delta_{q}\Big(T_{b}\frac{b^\theta}{r}\Big)=\sum_{\vert k-q \vert \le 4}\Delta_{q}\Big(S_{k-1}b\,\Delta_{k}\Big(\frac{b^\theta}{r}\Big)\Big).
$$
Therefore, we have
\begin{eqnarray}\label{hs1}
\Big\Vert T_{b}\frac{b^\theta}{r}\Big\Vert_{\widetilde{L}^1_t H^{s}}\lesssim\Big\| \Big(2^{ks}\Big\Vert S_{k-1}b\,\Delta_{k}\Big(\frac{b^\theta}{r}\Big)\Big\Vert_{{L}^1_tL^2}\Big)_k\Big\|_{\ell^2}
\end{eqnarray}
Since $b^r=0$ then we can easily check that
$$
\frac{b^\theta}{r}=-\frac{b^1}{x_2}=\frac{b^2}{x_1}.
$$
The term $S_{k-1}b\,\Delta_{k}\Big(\frac{b^\theta}{r}\Big)$ can be expanded under the form
\begin{align}\label{tbdecomp}
\nonumber S_{k-1}b^1(x)\Delta_{k}\Big(\frac{b^\theta}{r}\Big)&=-S_{k-1}b^1(x)\Delta_{k}\Big(\frac{b^1}{x_2}\Big)\\ \nonumber &=-x_2S_{k-1}\Big(\frac{b^1(x)}{x_2}\Big)\Delta_{k}\Big(\frac{b^1}{x_2}\Big)+x_2\Big[S_{k-1},\frac{1}{x_2}\Big]b^1(x)\Delta_{k}\Big(\frac{b^1}{x_2}\Big)\\
\nonumber &=-S_{k-1}\Big(\frac{b^1(x)}{x_2}\Big)\Delta_{k}b^1-x_2S_{k-1}\Big(\frac{b^1(x)}{x_2}\Big)\Big[\Delta_{k},\frac{1}{x_{2}}\Big]b^1\\ 
\nonumber&+x_2\Big[S_{k-1},\frac{1}{x_2}\Big]b^1(x)\Delta_{k}\Big(\frac{b^1}{x_2}\Big)\\
&:=-\textnormal{I}_{k}-\textnormal{II}_{k}+\textnormal{III}_{k}.
\end{align}
For the first term, we immediately get 
\begin{eqnarray}\label{hz80}
\Big\| \Big(2^{ks}\Vert\textnormal{I}_{k}\Vert_{{L}^1_tL^2}\Big)_k\Big\|_{\ell^2}
&\lesssim& \Big\Vert\frac{b^1}{x_2} \Big\Vert_{{L}^\infty_t L^\infty}\Vert b^1\Vert_{\widetilde{L}^1_t H^s} .
\end{eqnarray}
For the commutator term $\textnormal{II}_{k}$, we write by definition
\begin{eqnarray} \label{ii2}
\nonumber\textnormal{II}_{k}
 &=& S_{k-1}\Big(\frac{b^1(x)}{x_2}\Big)x_2\bigg(\Delta_{k}\Big(\frac{b^1(x)}{x_2}\Big)-\frac{\Delta_{k}b^1}{x_2}\bigg)\\
\nonumber&=& S_{k-1}\Big(\frac{b^1(x)}{x_2}\Big) 2^{3k}\int_{\R^3}h(2^{k}(x-y))(x_2-y_2) \frac{b^1}{y_2}(y) dy\\
&=& 2^{-k}S_{k-1}\Big(\frac{b^1(x)}{x_2}\Big)\Big(2^{3k}\widetilde{h}(2^{k}.)\ast\Big(\frac{b^1}{y_2} \Big)(x)\Big),
\end{eqnarray}
where $\widetilde{h}(x)=x_2 h(x).$
Since $\mathcal{F}(\widetilde{h}(\xi))=i\partial_{\xi_{2}}\mathcal{F}(h(\xi))=i\partial_{\xi_{2}}\varphi(\xi).$ Then it follows that $\textnormal{Supp}\;\mathcal{F}(\widetilde{h})\subset \textnormal{Supp}\,\mathcal{F}(h)=\textnormal{Supp}\,\varphi.$ Therefore for every $f \in \mathcal{S}^\prime$ we have $$2^{3k}\widetilde{h}(2^{k}.)\ast\Delta_{j}f=0\;,\;\textnormal{for}\;\vert k-j\vert\ge2.$$ This leads to
\be\label{ii22}
2^{3k}\widetilde{h}(2^{k}.)\ast f=\sum_{\vert k-j \vert \le 1}2^{3k}\widetilde{h}(2^{k}.)\ast \Delta_{j} f.
\ee
Hence by Young inequality for convolution 
we get
\begin{align*}
\Vert\textnormal{II}_{k}\Vert_{{L}^1_tL^2}&\lesssim 2^{-k}\Big\Vert S_{k-1}\Big(\frac{b^1(x)}{x_2}\Big)\Big\Vert_{{L}^\infty_tL^\infty}\sum_{\vert k-j \vert \le 1}\Big\Vert 2^{3k}\widetilde{h}(2^{k}.)\ast\Delta_{j}\Big(\frac{b^1}{x_2}\Big)\Big\Vert_{{L}^1_tL^2}
 \\
&\lesssim 2^{-k}\Vert\widetilde{h}\Vert_{L^1} \Big\Vert \frac{b^1}{x_2} \Big\Vert_{{L}^\infty_tL^\infty} \sum_{\vert k-j \vert \le 1}\Big\Vert\Delta_{j}\Big(\frac{b^1}{x_2}\Big)\Big\Vert_{L^1_tL^2}.
\end{align*}
It follows that 
\begin{eqnarray}\label{hz81}
 \nonumber \Big\| \Big(2^{ks}\Vert\textnormal{II}_{k}\Vert_{{L}^1_tL^2}\Big)_k\Big\|_{\ell^2}&\lesssim& \Big\Vert \frac{b^1}{x_2} \Big\Vert_{{L}^\infty_tL^\infty} \Big\| \Big(2^{k(s-1)} \sum_{\vert k-j \vert \le 1}\Big\Vert \Delta_{j}\Big(\frac{b^1}{x_2}\Big)\Big\Vert_{{L}^1_t L^2}\Big)_k\Big\|_{\ell^2}\\
\nonumber &\lesssim& \Big\Vert \frac{b^1}{x_2} \Big\Vert_{{L}^\infty_t L^\infty}\Big\| \Big( \sum_{\vert k-j \vert \le 1}2^{(k-j)(s-1)}2^{j(s-1)}\Big\Vert \Delta_{j}\Big(\frac{b^1}{x_2}\Big)\Big\Vert_{{L}^1_t L^2}\Big)_k\Big\|_{\ell^2}\\
&\lesssim& \Big\Vert \frac{b^1}{x_2}\Big\Vert_{{L}^\infty_tL^\infty}\Big\Vert\frac{b^1}{x_2}\Big\Vert_{\widetilde{L}^1_t H^{s-1}}. 
\end{eqnarray}
As for  the term $\textnormal{III}_{k}$, we write by definition
\begin{eqnarray} \label{iii3}
\nonumber\textnormal{III}_{k}&=&\Delta_{k}\Big(\frac{b^1(x)}{x_2}\Big)x_2\bigg(S_{k-1}\Big(\frac{b^1(x)}{x_2}\Big)-\frac{S_{k-1}b^{1}(x)}{x_2}\bigg)\\
\nonumber &=& \Delta_{k}\Big(\frac{b^1(x)}{x_2}\Big) 2^{3(k-1)}\int_{\R^3}g(2^{(k-1)}(x-y))(x_2-y_2) \frac{b^1}{y_2}(y) dy\\
&=& 2^{-k+1}\Delta_{k}\Big(\frac{b^1(x)}{x_2}\Big) \bigg(2^{3(k-1)}\widetilde{g}(2^{k-1}.)\ast\Big(\frac{b^1}{x_2} \Big)(x)\bigg),
\end{eqnarray}
where $\widetilde{g}(x)=x_2 g(x).$
By Young inequality for convolution 
we get
\begin{eqnarray}\label{hz82}
\nonumber \Vert\textnormal{III}_{k}\Vert_{{L}^1_t L^2}&\lesssim& 2^{-k}\Big\Vert \Delta_{k}\Big(\frac{b^1}{x_2}\Big)\Big\Vert_{{L}^1_t L^2}\Big\Vert 2^{3k}\widetilde{g}(2^{k}.)\ast\Big(\frac{b^1}{x_2}\Big)\Big\Vert_{{L}^\infty_t L^\infty}\\
\nonumber &\lesssim& 2^{-k}\Vert\widetilde{g}\Vert_{L^1} \Big\Vert \frac{b^1}{x_2}\Big\Vert_{{L}^\infty_t L^\infty}2^{-k}\Big\Vert\Delta_{k}\Big(\frac{b^1}{x_2}\Big)\Big\Vert_{{L}^1_t L^2}.
\end{eqnarray}
This yields 
\begin{eqnarray}\label{hz82}
\nonumber   \Big\| \Big(2^{ks}\Vert\textnormal{III}_{k}\Vert_{{L}^1_t L^2}\Big)_k\Big\|_{\ell^2}&\lesssim& \Big\Vert \frac{b^1}{x_2}\Big\Vert_{{L}^\infty_t L^\infty} \Big\| \Big(2^{k(s-1)}\Big\Vert\Delta_{k}\Big(\frac{b^1}{x_2}\Big)\Big\Vert_{{L}^1_t L^2}\Big)_k\Big\|_{\ell^2}\\
&\lesssim& \Big\Vert \frac{b^1}{x_2} \Big\Vert_{{L}^\infty_t L^\infty}\Big\Vert\frac{b^1}{x_2}\Big\Vert_{\widetilde{L}^1_t H^{s-1}}. 
\end{eqnarray}
Thus, from  \eqref{tbdecomp} \eqref{hz80}, \eqref{hz81} and \eqref{hz82} we deduce   that
\begin{eqnarray}\label{1st-part}
\Big\| \Big(2^{ks}\Big\Vert S_{k-1}b^1\,\Delta_{k}\Big(\frac{b^\theta}{r}\Big)\Big\Vert_{{L}^1_t L^2}\Big)_k\Big\|_{\ell^2}&\lesssim&  \Big\Vert \frac{b^1}{x_2} \Big\Vert_{{L}^\infty _t L^\infty}\bigg(\Big\Vert\frac{b^1}{x_2}\Big\Vert_{\widetilde{L}^1_t H^{s-1}}+\Vert b^1\Vert_{\widetilde{L}^1_t H^s}\bigg).
\end{eqnarray}
Since $b$ is axisymmetric then $b^{1}(x_1,0,z)=0.$ Therefore from Taylor formula,
\begin{equation*}
 b^1(x_1, x_2, z)=x_{2}\int_{0}^{1}\Big(\partial_{x_2}  b^1\Big)( x_1, \tau  x_2, z) d\tau.
\end{equation*}
It follows that
\begin{align*}
\Big\Vert \frac{ b^1(x)}{x_2}\Big\Vert_{\widetilde{L}^1_t  H^{s-1}}
&\lesssim \int_0^1\Vert (\partial_{x_2}  b^1)(.,\tau,.) \Vert_{\widetilde{L}^1_t  H^{s-1}}d\tau\\
& \lesssim \Vert \partial_{x_2}  b^1 \Vert_{\widetilde{L}^1_t  H^{s-1}}\int_0^1\tau^{s-\frac52}d\tau
\\ & \lesssim \Vert b^1 \Vert_{\widetilde{L}^1_t  H^{s}}.  
\end{align*}
Inserting the last estimate into \eqref{1st-part} yields
\begin{eqnarray}\label{z77}
 \Big\| \Big(2^{ks}\Big\Vert S_{k-1}b^1\,\Delta_{k}\Big(\frac{b^\theta}{r}\Big)\Big\Vert_{{L}^1_t L^2}\Big)_k\Big\|_{\ell^2}&\lesssim&  \Big\Vert \frac{b^1}{x_2} \Big\Vert_{{L}^\infty_t L^\infty}\Vert b^1 \Vert_{\widetilde{L}^1_t  H^s}. 
\end{eqnarray}
In a similar way we can show that
\begin{eqnarray}\label{hz9}
\Big\| \Big(2^{ks}\Big\Vert S_{k-1}b^2\,\Delta_{k}\Big(\frac{b^\theta}{r}\Big)\Big\Vert_{{L}^1_t  L^2}\Big)_k\Big\|_{\ell^2}&\lesssim&  \Big\Vert \frac{b^2}{x_1} \Big\Vert_{{L}^\infty_t  L^\infty}\Vert b^2\Vert_{\widetilde{L}^1_t  H^s}.
\end{eqnarray}
Combining \eqref{z77}, \eqref{hz9}  and  \eqref{hs1} we conclude  that
\begin{eqnarray}\label{hz10}
 \Big\Vert T_{b}\frac{b^\theta}{r}\Big\Vert_{\widetilde{L}^1_t H^s}&\lesssim&\Big\Vert \frac{b^\theta}{r} \Big\Vert_{{L}^\infty_t L^\infty}\Vert b\Vert_{\widetilde{L}^1_t H^s}.
\end{eqnarray}
Now putting  together \eqref{hs1}, \eqref{hz4}, \eqref{hz10} and \eqref{hz2} we find
\begin{equation*}
\Big\Vert \frac{b^\theta}{r}b\Big\Vert_{\widetilde{L}^1_t H^s}\lesssim \Big\Vert \frac{b^\theta}{r} \Big\Vert_{{L}^\infty_t L^\infty}\Vert b\Vert_{\widetilde{L}^1_t H^s}. 
\end{equation*}
This concludes the proof of (i).

\smallskip

To prove (ii) we use \eqref{hz2} to write
\begin{eqnarray}\label{z22b}
\Big\Vert\frac{b^\theta}{r} b\Big\Vert_{B^{\frac3p+1}_{p,1}}&\le& \big\Vert T_{\frac{b^\theta}{r}}b\big\Vert_{B^{\frac3p+1}_{p,1}}+\Big\Vert T_{b}\frac{b^\theta}{r}\Big\Vert_{B^{\frac3p+1}_{p,1}}+ \Big\Vert \mathcal{R}\Big(\frac{b^\theta}{r},b\Big)\Big\Vert_{B^{\frac3p+1}_{p,1}}.
\end{eqnarray}
In view of  \eqref{tgood} and by H\"older and Bernstein inequalities  we get
\begin{eqnarray}\label{s1b}
\nonumber\big\Vert T_{\frac{b^\theta}{r}}b\big\Vert_{B^{\frac3p+1}_{p,1}}&\lesssim &\sum_{q\in \mathbb{N}}2^{q(\frac3p+1)}\Big\Vert S_{q-1}\frac{b^\theta}{r}\Delta_{q}b\Big\Vert_{L^p}\\
&\lesssim& \Big\Vert \frac{b^\theta}{r}\Big\Vert_{L^\infty} \Vert b\Vert_{B^{\frac3p+1}_{p,1}}.
\end{eqnarray}
For the remainder term, by \eqref{gr},  we have
\begin{eqnarray}\label{z4b}
\nonumber \Big\Vert\mathcal{R}\Big(b,\frac{b^\theta}{r}\Big)\Big\Vert_{B^{\frac3p+1}_{p,1}}&\lesssim& \sum_{j\in\mathbb{N}}2^{j(\frac3p+1)}\sum_{k\ge j-4}\Vert\Delta_{k}b\Vert_{L^\infty}\Big\Vert\widetilde{\Delta}_{k}\Big(\frac{b^\theta}{r}\Big)\Big\Vert_{L^{\infty}}\\
\nonumber &\lesssim& \Big\Vert\frac{b^\theta}{r}\Big\Vert_{L^{\infty}}\sum_{k}2^{k(\frac3p+1)}\Vert\Delta_{k}b\Vert_{L^\infty}\sum_{j\le k+4}2^{(j-k)(\frac3p+1)}\\
 &\lesssim& \Big\Vert\frac{b^\theta}{r}\Big\Vert_{L^{\infty}}\Vert b \Vert_{B^{\frac3p+1}_{p,1}}
\end{eqnarray}
To estimate the term $\Big\Vert T_{b}\frac{b^\theta}{r}\Big\Vert_{B^{1}_{\infty,1}},$ we write by definition
\begin{equation}\label{tb-bes}
\Big\Vert T_{b}\frac{b^\theta}{r}\Big\Vert_{B^{1}_{\infty,1}}\lesssim \sum_{k \ge 0} 2^{k}\Big\Vert S_{k-1}b(x)\Delta_{k}\Big(\frac{b^\theta}{r}\Big)\Big\Vert_{L^{\infty}}.
\end{equation}
In view of \eqref{tbdecomp},  \eqref{ii2}, \eqref{ii22} and \eqref{iii3} the term $S_{k-1}b^1(x)\Delta_{k}\Big(\frac{b^\theta}{r}\Big)$ can be decomposed as follow,
\begin{align}\label{bt}
S_{k-1}b^1(x)\Delta_{k}\Big(\frac{b^\theta}{r}\Big)&=
-\textnormal{I}_{k}-\textnormal{II}_{k}+\textnormal{III}_{k},
\end{align}
with
\begin{align*}
\textnormal{I}_{k}&=S_{k-1}\Big(\frac{b^1(x)}{x_2}\Big)\Delta_{k}b^1,\\
\textnormal{II}_{k}&=2^{-k}S_{k-1}\Big(\frac{b^1(x)}{x_2}\Big)\sum_{\vert k-q \vert \le 1}\Big(2^{3k}\widetilde{h}(2^{k}.)\ast \Delta_q\Big(\frac{b^1}{y_2} \Big)(x)\Big),\\
\textnormal{III}_{k}&=2^{-k+1}\Delta_{k}\Big(\frac{b^1(x)}{x_2}\Big) \bigg(2^{3(k-1)}\widetilde{g}(2^{k-1}.)\ast\Big(\frac{b^1}{x_2} \Big)(x)\bigg).
\end{align*}
where $\widetilde{h}(x)=x_2 h(x)$ and $\widetilde{g}(x)=x_2 g(x).$
For the first term $\textnormal{I}_{k}$, we   have
\begin{eqnarray}\label{z80}
\nonumber \sum_{k \ge 0}2^{k(\frac3p+1)}\Vert\textnormal{I}_{k}\Vert_{L^p}&\lesssim& \Big\Vert\frac{b^1}{x_2} \Big\Vert_{L^\infty}\sum_{k \ge 0}2^{k}\Vert\Delta_{k}b^{1}\Vert_{L^p}\\
&\lesssim& \Big\Vert\frac{b^1}{x_2} \Big\Vert_{L^\infty}\Vert b^1\Vert_{B^{\frac3p+1}_{p,1}} 
\end{eqnarray}
For the second  $\textnormal{II}_{k}$ and the third term $\textnormal{III}_{k}$,  we shall make use of Young inequality for convolution,   
\begin{eqnarray}\label{z81}
\nonumber \sum_{k \ge 0}2^{k(\frac3p+1)}\Vert\textnormal{II}_{k}\Vert_{L^p}
\nonumber &\lesssim& \Big\Vert \frac{b^1}{x_2} \Big\Vert_{L^\infty}\sum_{\vert k-q\vert \le 1}2^{q\frac3p}\Vert\widetilde{h}\Vert_{L^1}\Big\Vert\Delta_{q}\Big(\frac{b^1}{x_2}\Big)\Big\Vert_{L^p}\\
&\lesssim& \Big\Vert \frac{b^1}{x_2}\Big\Vert_{L^\infty}\Big\Vert\frac{b^1}{x_2}\Big\Vert_{B^{\frac3p}_{p,1}},
\end{eqnarray}
\begin{eqnarray}\label{z82}
\nonumber \sum_{k \ge 0}2^{k(\frac3p+1)}\Vert\textnormal{III}_{k}\Vert_{L^p}
\nonumber &\lesssim& \Vert\widetilde{g}\Vert_{L^1} \Big\Vert \frac{b^1}{x_2}\Big\Vert_{L^\infty}\sum_{k\geq 0}2^{k\frac3p}\Big\Vert\Delta_{k}\Big(\frac{b^1}{x_2}\Big)\Big\Vert_{L^p}\\
&\lesssim& \Big\Vert \frac{b^1}{x_2} \Big\Vert_{L^\infty}\Big\Vert\frac{b^1}{x_2}\Big\Vert_{B^{\frac3p}_{p,1}}. 
\end{eqnarray}
Thus, it follows from \eqref{z80}, \eqref{z81} and \eqref{z82}  that
\begin{eqnarray}
\nonumber \sum_{k \ge 0} 2^{k(\frac3p+1)}\Big\Vert S_{k-1}b^1(x)\Delta_{k}\Big(\frac{b^\theta}{r}\Big)\Big\Vert_{L^{p}}&\lesssim&  \Big\Vert \frac{b^1}{x_2} \Big\Vert_{L^\infty}\bigg(\Big\Vert\frac{b^1}{x_2}\Big\Vert_{B^{\frac3p}_{p,1}}+\Vert b^1\Vert_{B^{\frac3p+1}_{p,1}}\bigg).
\end{eqnarray}
Since $b$ is axisymmetric then $b^{1}(x_1,0,z)=0.$ Therefore from Taylor formula,
\begin{equation*}
 b^1(x_1, x_2, z)=x_{2}\int_{0}^{1}\Big(\partial_{x_2}  b^1\Big)( x_1, \tau  x_2, z) d\tau.
\end{equation*}
For $1\leq p<\infty$ one has
\begin{align*}
\Big\Vert \frac{ b^1(x)}{x_2}\Big\Vert_{B^{\frac3p}_{p,1}}
& \lesssim \Vert \partial_{x_2}  b^1 \Vert_{B^{\frac3p}_{p,1}}
\\ & \lesssim \Vert b^1 \Vert_{B^1_{\infty,1}}.  
\end{align*}
For $p=\infty$ we use  the estimate \eqref{prop-a5},
\begin{align*}
\Big\Vert \frac{ b^1(x)}{x_2}\Big\Vert_{B^0_{\infty,1}}
& \lesssim \Vert \partial_{x_2}  b^1 \Vert_{B^0_{\infty,1}}\int_0^1(1-\log \tau)d\tau
\\ & \lesssim \Vert b^1 \Vert_{B^1_{\infty,1}}.  
\end{align*}
Consequently, 
\begin{eqnarray}\label{z7}
 \sum_{k \ge 0} 2^{k(\frac3p+1)}\Big\Vert S_{k-1}b^1(x)\Delta_{k}\Big(\frac{b^\theta}{r}\Big)\Big\Vert_{L^{p}}&\lesssim&  \Big\Vert \frac{b^1}{x_2} \Big\Vert_{L^\infty}\Vert b^1 \Vert_{B^{\frac3p+1}_{p,1}}. 
\end{eqnarray}
In a similar way we can prove the same estimate for the second component,
\begin{eqnarray}\label{z9}
\sum_{k \ge 0} 2^{k(\frac3p+1)}\Big\Vert S_{k-1}b^2(x)\Delta_{k}\Big(\frac{b^\theta}{r}\Big)\Big\Vert_{L^{p}}&\lesssim&  \Big\Vert \frac{b^2}{x_1} \Big\Vert_{L^\infty}\Vert b^2 \Vert_{B^{\frac3p+1}_{p,1}}.
\end{eqnarray}
It follows from \eqref{tb-bes},  \eqref{z7} and \eqref{z9} that
\begin{eqnarray}\label{z10b}
 \Big\Vert T_{b}\frac{b^\theta}{r}\Big\Vert_{B^{\frac3p+1}_{p,1}}&\lesssim&\Big\Vert \frac{b^\theta}{r} \Big\Vert_{L^\infty}\Vert b \Vert_{B^{\frac3p+1}_{p,1}}.
\end{eqnarray}
Now putting  together \eqref{z22b}, \eqref{s1b}, \eqref{z4b} and \eqref{z10b} we find
\begin{equation*}
\Big\Vert \frac{b^\theta}{r}b\Big\Vert_{B^{\frac3p+1}_{p,1}}\lesssim \Big\Vert \frac{b^\theta}{r} \Big\Vert_{L^\infty}\Vert b \Vert_{B^{\frac3p+1}_{p,1}}. 
\end{equation*}
This ends the proof of the proposition.
\end{proof}

\vspace{0.2cm}

{\bf Acknowledgments}
The author would like to thank T. Hmidi for the fruitful discussions about the critical case.


\Addresses


\begin{thebibliography}{99}

{\small  \bibitem{Al} G.V. Alexseev, 
 Solvability of a homogeneous initial-boundary value problem for equations of
magneto-hydrodynamics of an ideal fluid, (Russian), 
  \emph{Dinam. Sploshn. Sredy,}
  {\bf 57} (1982), 3--20.

 \bibitem{A} H. Alfv\'en,  
 Existence of electromagnetic-hydrodynamic waves,
 \emph{ Nature},
  {\bf 150} (1942), 3805, 405--406.
 
 \bibitem{AHK} H. Abidi,  T. Hmidi,  and S. Keraani, S. 
 On the global well-posedness for the axisymmetric Euler equations. 
   \emph{Math. Ann.}
   {\bf 347} (2010),  15--41
   
 \bibitem{BCD}   H. Bahouri, J. -Y. Chemin and R. Danchin, 
 Fourier Analysis and Nonlinear Partial Differential Equations, 
  \emph{ Grundlehren der Mathematischen Wissenschaften},
  {\bf 343} (2011).
  
  \bibitem{bkm84} J. T. Beale, T. Kato and A. Majda,
   {Remarks on the Breakdown of Smooth Solutions for the 3-D Euler Equations}.
   \emph{Commun. Math. Phys.} 
  \textbf{94} (1984), p.61--66 .
  
  \bibitem{B}  D. Biskamp, Nonlinear Magnetohydrodynamics, Cambridge University Press, Cambridge, 1993.
  
  \bibitem{bo81} J.-M. Bony, {Calcul symbolique et propagation des singularit\'es pour les \'equations aux d\'eriv\'ees partielles non lin\'eaires},
   \emph{Ann. de l'Ecole Norm. Sup.,} 
  \textbf{14}, (1981) p. 209--246, .
  
  \bibitem{b76} J. Bergh, J. L\"ofstr\"om, {Interpolation spaces. An introduction}. \emph{Springer-Verlage}, (1976).
  
   \bibitem{CK}   R. E. Caflisch, I. Klapper, and G. Steele. Remarks on singularities, dimension and energy dissipation for
ideal hydrodynamics and MHD. 
  \emph{ Comm. Math. Phys.}, {\bf 184} (1997), 443--455.
  
  \bibitem{CW}  C. Cao and J. Wu, Global regularity for the 2D MHD equations with mixed partial dissipation
and magnetic diffusion,
  \emph{ Adv. Math.},
  {\bf 226} (2011), 1803--1822.
  
 \bibitem{CWY} C.  Cao,  J. Wu and B. Yuan,
  The 2D incompressible magnetohydrodynamics equations with only magnetic diffusion
  \emph{ SIAM J. Math. Anal.}{ \bf 46} (2014) 588--602.
  
   \bibitem{CST} E. Casella; P. Secchi, P. Trebeschi. Global classical solutions for MHD system. 
   {\bf  J. Math. Fluid Mech. 5}
(2003),{ \bf 1}, 70--91.
  
  
  
   \bibitem{Ch}   S. Chandrasekhar. Hydrodynamic and hydromagnetic stability. 
   The International Series of Monographs on
Physics Clarendon Press, Oxford 1961
  
  \bibitem{che98} J.-Y. Chemin,
  { Perfect incompressible fluids}. {Oxford University Press} (1998).
  
  \bibitem{CMRR}  J.-Y. Chemin, D. S. McCormick, J. C. Robinson and  J. L. Rodrigo ,
   Local existence for the non-resistive MHD equations in Besov spaces,
   \emph{Advances in Mathematics}
   {\bf 286} (2016), 1--31
   
  \bibitem{CMZ}   Q. Chen, C. Miao, Z. Zhang. On the well-posedness of the ideal MHD equations in the Triebel-Lizorkin
spaces. 
\emph{Arch. Ration. Mech. Anal.}
{\bf 195} (2010),  561--578.

  
  \bibitem{d07} R. Danchin : {Axisymmetric incompressible flows with bounded vorticity},  \emph{ Russian Math. Surveys} \textbf{62} (2007), 73--94.
   
  \bibitem{D}  R. Danchin, 
  The inviscid limit for density-dependent incompressible fluids,
  \emph{ Ann. Fac. Sci. Toulouse Math.} 
  {\bf 15 } (2006), 637--688.
  
  \bibitem{Dav}  P.A. Davidson, 
  An Introduction to Magnetohydrodynamics, Cambridge Uni-
versity Press, Cambridge, England, 2001.
  
  
  
   \bibitem{DL}  G. Duvaut and J.-L. Lions, 
 In\'equations en thermo\'elasticit\'e et magn\'eto- hydrodynamique, 
   \emph{ Arch. Rational Mech. Anal.}
     {\bf 46}  (1972), 241--279.
     
       \bibitem{FO} J. Fan,  and T. Ozawa,  , Regularity criteria for the magnetohydrodynamic equations with partial viscous terms and the Leray-$\alpha$-MHD model, \emph{Kinet. Relat. Models}, {\bf 2}(2009), 293--305.
     
    
 \bibitem{FMRR}  C. L. Fefferman,  D. S. McCormick, J. C. Robinson and  J. L. Rodrigo , 
 Higher order commutator estimates and local existence for the non-resistive MHD equations and related models,
 \emph{ J. Funct. Anal.} 
 {\bf 267} (2014) , 1035--1056.
 
  \bibitem{FMRR2}  C. L. Fefferman,  D. S. McCormick, J. C. Robinson and  J. L. Rodrigo , 
 Local existence for the non-resistive MHD equations in nearly optimal Sobolev spaces
   \emph{ Arch. Rational Mech. Anal.}
 {\bf 223} (2017),  677--691.
 
  
 \bibitem{H}  T. Hmidi, 
  The low Mach number limit for the isentropic Euler system with axisymmetric
initial data. 
  \emph{J. Inst. Math. Jussieu }
  {\bf12} (2013),  385--389.
  
  \bibitem{HKB} T. Hmidi and S. Keraani, 
   Incompressible Viscous Flows in Borderline Besov Spaces. 
     \emph{Arch Rational Mech Anal} 
     {\bf189},  (2008) 283--300

 \bibitem{HK} T. Hmidi, S. Keraani, 
 On the global well-posedness of the Boussinesq system with zero viscosity, 
  \emph{Indiana Univ. Math. J.}
{\bf 58} (2009), 1591--1618.

 \bibitem{HR} T. Hmidi, F. Rousset, 
 Global well-posedness for the Euler–Boussinesq system with axisymmetric data, 
 \emph{J. Funct. Anal.}
 {\bf 260} (2011), 745--796.
 
  \bibitem{JN}  Q. Jiu,  and D. Niu, 
  Mathematical results related to a two-dimensional magneto-hydrodynamic equations, 
  \emph{Acta Math. Sci. Ser. B Engl. Ed.} 
  {\bf 26} (2006), 744--756.
  
  \bibitem{JZ}  Q. Jiu  and J. Zhao,
     Global regularity of 2D generalized MHD equations with magnetic diffusion
\emph{Z. Angew. Math. Phys.} {\bf 66} (2015), 677--87.
  
   \bibitem{K} H. Kozono. Weak and classical solutions of the two-dimensional magnetohydrodynamic equations. 
   \emph{Tohoku Math. J.} {\bf 41} (1989),  471--488.
   
  \bibitem{Lei} Z. Lei, 
  On axially symmetric incompressible magnetohydrodynamics in three dimensions
  \emph{  J. Differential Equations} 
    {\bf  259} (2015), 3202--3215


  
  \bibitem{LZ} Z. Lei and Y. Zhou, BKM's criterion and global weak solutions for magnetohydrodynamics
with zero viscosity, 
  \emph{Discrete Contin. Dyn. Syst.},
  {\bf 25} (2009),  575--583.
  
 \bibitem{L} P.-G. Lemari\'e, Recent Developments in the Navier–Stokes Problem, CRC Press, 2002.
 
  \bibitem{Lau} L.D. Laudau, E.M. Lifshitz, Electrodynamics of Continuous Media, 2nd ed., Pergamon, NewYork, 1984.
  
   \bibitem{MY} C.  Miao and B. Yuan,  
   Well-posedness of the ideal MHD system in critical Besov spaces,
    \emph{Methods Appl. Anal.} 
   {\bf 13} (2006), 89--106.
 
 \bibitem{Pee}  J. Peetre, 
 New thoughts on Besov spaces, Duke University Mathematical Series 1, Durham N. C. (1976).
 
 \bibitem{Sc}  P.G. Schmidt, 
 On a magneto-hydrodynamic problem of Euler type, 
  \emph{  J. Diff. Eq.},
 {\bf 74} (1988), 318--335.
 
 \bibitem{Se} P. Secchi, On the equations of ideal incompressible magneto-hydrodynamics, 
  \emph{ Rend. Sem. Mat. Univ. Padova,}
  {\bf 90} (1993), 103--119.
  
 \bibitem{Serfati}  Serfati, P.: Solutions $C^\infty$ en temps, $n$-log Lipschitz born\'ees en espace et \'equation d’Euler.  \emph{C. R. Acad. Sci. Paris Sér. I Math.} {\bf 320}(1995), 555--558.
 
 \bibitem{ST} M. Sermange and  R.  Temam,
  Some mathematical questions related to the MHD equations,
  \emph{ Comm. Pure Appl. Math.} 
  {\bf 36} (1983), 635--664.
 
 \bibitem{S} S. Sulaiman
 On the global existence for the axisymmetric Euler-Boussinesq system in critical Besov space,
  \emph{ Asymptot. Anal.}
{\bf 77} (2012),  89--121

 \bibitem{SY} T. Shirota and T. Yanagisawa, 
 Note on global existence for axially symmetric solutions of the Euler system, 
 \emph{Proc. Japan Acad. Ser. A Math. Sci.}, 
 {\bf 70} (1994), 299--304.
 
 \bibitem{ui68} M. R. Uhkovskii, V. I. Iudovich, 
 {Axially symmetric flows of ideal and viscous fluids filling the whole space},  
 \emph{Prikl. Mat. Meh.} 
 \textbf{32} (1968),  59--69. 
 
 \bibitem{vis98} M. Vishik, {Hydrodynamics in Besov spaces}. 
 \emph{Arch. Rational Mech. Anal.}, \textbf{145}, (1998), 197--214.
 
 \bibitem{W} J. Wu. Viscous and inviscid magnetohydrodynamics equations.
  \emph{J. Anal. Math.} {\bf 73} (1997), 251--265.

}
\end{thebibliography}
\end{document}